\documentclass[a4paper,11pt]{amsart}
\usepackage{amsmath}
\usepackage{amsthm,amsfonts,amssymb,graphics}
\usepackage[foot]{amsaddr}
\usepackage{verbatim}
\usepackage{hyperref}
\usepackage{xcolor}
\usepackage[normalem]{ulem}
\usepackage{enumerate}
\usepackage{geometry}
\usepackage{dsfont}
\usepackage{centernot}
\usepackage{bbm}
\usepackage{tikz}
\usetikzlibrary{patterns,calc,arrows.meta}

\newtheorem{theorem}{Theorem}
\newtheorem{corollary}{Corollary}
\newtheorem{lemma}{Lemma}
\newtheorem{proposition}{Proposition}

\theoremstyle{remark}
\newtheorem{assumption}{Assumption}
\newtheorem{remark}{Remark}
\newtheorem{question}{Question}
\newtheorem{definition}{Definition}

\newcommand \id{\mathds 1}
\newcommand {\R} {\mathbb{R}}

\newcommand {\E} {\mathbb{E}}
\newcommand {\N} {\mathbb{N}}

\renewcommand{\P} {\mathbb{P}}
\newcommand {\Var} {\mathrm{Var}}

\newcommand {\Cov} {\mathrm{Cov}}
\newcommand {\DC} {\mathrm{DC}}

\newcommand {\eps}{\varepsilon}

\newcommand{\ind}{\mathds{1}}

\begin{document}
\title[A covariance formula for excursion set components and applications]{A covariance formula for the number of excursion set components of Gaussian fields and applications} 
\author{Dmitry Beliaev\textsuperscript{1}}
\address{\textsuperscript{1}Mathematical Institute, University of Oxford}
\email{belyaev@maths.ox.ac.uk}
\author{Michael McAuley\textsuperscript{2}}
\address{\textsuperscript{2}Department of Mathematics and Statistics, University of Helsinki} 
\email{m.mcauley@cantab.net}
\author{Stephen Muirhead\textsuperscript{3}}
\address{\textsuperscript{3}School of Mathematics and Statistics, University of Melbourne}
\email{smui@unimelb.edu.au}
\begin{abstract}
We derive a covariance formula for the number of excursion or level set components of a smooth stationary Gaussian field on $\R^d$ contained in compact domains. We also present two applications of this formula: (1) for fields whose correlations are integrable we prove that the variance of the component count in large domains is of volume order and give an expression for the leading constant, and (2) for fields with slower decay of correlation we give an upper bound on the variance which is of optimal order if correlations are regularly varying, and improves on best-known bounds if correlations are oscillating (e.g.\ monochromatic random waves).
\end{abstract}
\date{\today}
\thanks{}

\maketitle

\section{Introduction}\label{s:introduction}
Let $f:\R^d\to\R$ be a smooth stationary centred Gaussian field, let $\ell \in \R$, and for a compact domain $D \subset \R^d$ let $N_\mathrm{ES}(D)$ and $N_\mathrm{LS}(D)$ denote respectively the number of connected components of the excursion set $\{x \in \R^d : f(x) \geq\ell\}$ and the level set $\{x \in \R^d : f(x) = \ell\}$ which are contained in $D$ (i.e.\ which intersect $D$ but not its boundary). We abbreviate $N_\star(D)$ for $\star \in \{\mathrm{ES},\mathrm{LS}\}$, and refer to these variables generically as the `component count'.
 
Recently, the asymptotic distribution of the component count in large domains has attracted considerable attention, motivated by potential applications (in, e.g., cosmology \cite{bbks86,pra19}, medical imaging \cite{wor96}, quantum chaos~\cite{js17}), as well as by connections to other areas such as percolation theory \cite{bs02}.

By now the first-order asymptotics of the component count are well-understood: in \cite{ns16} Nazarov and Sodin established that, under very general conditions, the component count satisfies a strong law of large numbers with volume-order scaling. Although the limiting constant is defined implicitly, its properties have been studied in, e.g., \cite{kw18,bmm20,bmm20a}.

A natural next step is to investigate second-order properties. Unlike the law of large numbers, it is expected that the order of the variance depends strongly on the covariance structure of the field. In Section \ref{s:known} below we give a brief overview of known and conjectured second-order results.

In this paper we establish an exact covariance formula for the component counts on two compact domains. This formula is inspired by an analogous covariance formula for `topological events' that appeared in \cite{bmr20}, and is of a similar `interpolation' type to related covariance formulae for Gaussian vectors \cite{pit96,cha08,cha14} and fields \cite{bmr20,mrv23}.

As an application of the covariance formula we prove new results for the variance of the component count. First, assuming correlations are integrable, we prove that the variance has volume-order scaling, and give an expression for the limiting constant (see Theorem \ref{t:Asymptotic}). Second, we give a general upper bound on the variance in terms of the covariance kernel of the field which in many cases improves on best-known results, e.g.\ the important case of monochromatic random waves, and is conjecturally optimal for fields with regularly varying correlations (see Theorem~\ref{t:vub} and Corollary~\ref{c:vf2}).

\subsection{A covariance formula for component counts}
\label{s:cf}

For the remainder of the paper we assume the following general conditions on the field $f$. Let $K(x) = \E[f(0)f(x)]$ be the \textit{covariance kernel} of $f$. Recall that the \textit{spectral measure} of $f$ is the finite measure $\mu$ on $\R^d$ such that $K$ is the Fourier transform of $\mu$. The choice of the normalization in the Fourier transform is unimportant for this paper. We will use the following definition: $K(x) = \mathcal{F}[\mu](x):=\int_{\R^d}e^{\mathrm{i}t\cdot x}\;\mu(dt)$.
\begin{assumption}[General conditions on the field] 
\label{a:gen} $\,$
\begin{enumerate}
\item $f$ is $C^{3}$-smooth almost surely;
\item As $|x| \to \infty$, $\max_{|\alpha| \le 2} \partial^\alpha K(x) \to 0$;
\item The support of the spectral measure $\mu$ contains either (i) an open set, or (ii) a scaled sphere $a \mathbb{S}^{d-1}$ for some $a > 0$.
\end{enumerate}
\end{assumption}

We actually prove our results using only a weaker version of the third item above -- see Remark \ref{r:nondegen} -- however the present statement is convenient and covers all of our examples. The second item of Assumption \ref{a:gen} is used in only one place (in the proof of Lemma \ref{l:numbound}) and could perhaps also be weakened. 

To state the covariance formula we introduce some notation. The formula will apply to the component count on \textit{boxes}, which are sets of the form $B = [a_1, b_1] \times \cdots \times [a_d, b_d] \subset \R^d$, for finite $a_i \le b_i$.  Henceforth we fix a level $\ell \in \R$, two boxes $B_1$ and $B_2$, and $\star,\diamond \in \{\mathrm{ES},\mathrm{LS}\}$.

We first introduce an \text{interpolation} between $f$ and an independent copy $\tilde{f}$
\begin{equation}
    \label{e:int}
 f^t := t f + \sqrt{1 - t^2}  \tilde{f}  \, , \quad t \in [0,1],
 \end{equation}
remarking that $f^t$ is equal in law to $f$ for each $t \in [0, 1]$, and that $f^0 = \tilde{f}$ and $f^1 = f$. We denote by $N^t_\star(D) $ the analogue of $N_\star(D)$ for the field $f^t$.

We next define a notion of `pivotal measures' for the pair of fields $(f,f^t)$; to keep the discussion simple we defer precise definitions to Section \ref{s:vfproof} and instead convey the intuition. 

Consider conditioning the field $f$ to have a critical point at $x \in B_1$ with critical value $\ell$. Then we can classify this critical point as being either `positively pivotal', `negatively pivotal', or `not pivotal', depending on whether a small local positive perturbation of the field $f$ near the point $x$ increases the component count $N_\star(B_1)$ by one, decreases the component count $N_\star(B_1)$ by one, or keeps the component count the same (see Figure~\ref{fig:Pivotal1} for examples). This classification is well-defined for non-degenerate critical points (see Lemma~\ref{l:change}).

\begin{figure}[ht]
    \centering
\begin{tikzpicture}[scale=1.2]
\draw[dashed] (-1,-1) rectangle (1,1);
\node[left] at (-1,0.6) {$B_1$};
\draw[thick,pattern=north west lines,pattern color=gray] plot [smooth cycle,tension=0.5] coordinates {(-0.7,0)(-0.6,.2)(-0.4,0.3)(-0.2,.3)(0,0)(0.2,0.3)(.4,0.3)(0.6,0.2)(0.7,0)(0.6,-0.2)(.4,-0.3)(0.2,-0.3)(0,0)(-0.2,-0.3)(-0.4,-0.3)(-0.6,-0.2)};
\node[below] at (0.7,-1.2) {$\{f\geq\ell\}$};
\draw (0.7,-1.2)--(0.3,0);
\draw[fill] (0,0) circle (1pt);
\draw (0,0)--(-.5,-1.2);
\node[below] at (-0.5,-1.2) {$x$};

\draw[-{Triangle[width=18pt,length=8pt]}, line width=10pt](-1.2,0) -- (-2,0);
\draw[-{Triangle[width=18pt,length=8pt]}, line width=10pt](1.2,0) -- (2,0);

\begin{scope}[shift={(-3.2,0)}]
\draw[dashed] (-1,-1) rectangle (1,1);
\node[left] at (-1,0.6) {$B_1$};
\draw[thick,pattern=north west lines,pattern color=gray] plot [smooth cycle,tension=0.5] coordinates {(-0.7,0)(-0.6,.2)(-0.4,0.3)(-0.2,.3)(-0.1,0)(-0.2,-0.3)(-0.4,-0.3)(-0.6,-0.2)};
\draw[thick,pattern=north west lines,pattern color=gray] plot [smooth cycle,tension=0.5] coordinates {(0.1,0)(0.2,0.3)(.4,0.3)(0.6,0.2)(0.7,0)(0.6,-0.2)(.4,-0.3)(0.2,-0.3)};
\node[below] at (0,-1.2) {$\{f-\delta h\geq\ell\}$};
\draw (0,-1.2)--(-0.3,-0.2);
\draw (0,-1.2)--(0.3,-0.2);
\draw[dashed] (0,0) circle (5pt);
\node[above] at (0,.6) {$W$};
\draw (0,.6)--(0,0.25);
\end{scope}

\begin{scope}[shift={(3.2,0)}]
\draw[dashed] (-1,-1) rectangle (1,1);
\node[left] at (-1,0.6) {$B_1$};
\draw[thick,pattern=north west lines,pattern color=gray] plot [smooth cycle,tension=0.5] coordinates {(-0.7,0)(-0.6,.2)(-0.4,0.3)(-0.2,.3)(0,0.05)(0.2,0.3)(.4,0.3)(0.6,0.2)(0.7,0)(0.6,-0.2)(.4,-0.3)(0.2,-0.3)(0,-0.05)(-0.2,-0.3)(-0.4,-0.3)(-0.6,-0.2)};
\node[below] at (0,-1.2) {$\{f+\delta h\geq\ell\}$};
\draw (0,-1.2)--(0.3,-0.2);
\draw[dashed] (0,0) circle (5pt);
\node[above] at (0,.6) {$W$};
\draw (0,.6)--(0,0.25);
\end{scope}

\begin{scope}[shift={(0,-3.2)}]
\draw[dashed] (-1,-1) rectangle (1,1);
\node[left] at (-1,0.6) {$B_1$};
\draw[thick,pattern=north west lines,pattern color=gray] (0,0) circle (0.5);

\node[below] at (0.7,-1.2) {$\{f\geq\ell\}$};
\draw (0.7,-1.2)--(0.3,0);
\draw[fill] (0,0) circle (1pt);
\draw (0,0)--(-.5,-1.2);
\node[below] at (-0.5,-1.2) {$x$};

\draw[-{Triangle[width=18pt,length=8pt]}, line width=10pt](-1.2,0) -- (-2,0);
\draw[-{Triangle[width=18pt,length=8pt]}, line width=10pt](1.2,0) -- (2,0);

\begin{scope}[shift={(-3.2,0)}]
\draw[dashed] (-1,-1) rectangle (1,1);
\node[left] at (-1,0.6) {$B_1$};
\draw[thick,pattern=north west lines,pattern color=gray] (0,0) circle (0.5);
\draw[thick,fill=white] (0,0) circle (2pt);
\node[below] at (0,-1.2) {$\{f-\delta h\geq\ell\}$};
\draw (0,-1.2)--(0,-0.35);
\draw[dashed] (0,0) circle (5pt);
\node[above] at (0,.6) {$W$};
\draw (0,.6)--(0,0.25);
\end{scope}

\begin{scope}[shift={(3.2,0)}]
\draw[dashed] (-1,-1) rectangle (1,1);
\node[left] at (-1,0.6) {$B_1$};
\draw[thick,pattern=north west lines,pattern color=gray] (0,0) circle (0.5);
\node[below] at (0,-1.2) {$\{f+\delta h\geq\ell\}$};
\draw (0,-1.2)--(0,-0.35);
\draw[dashed] (0,0) circle (5pt);
\node[above] at (0,.6) {$W$};
\draw (0,.6)--(0,0.25);
\end{scope}
\end{scope}

\begin{scope}[shift={(0,-6.4)}]
\draw[dashed] (-1,-1) rectangle (1,1);
\node[left] at (-1,0.6) {$B_1$};
\draw[thick,pattern=north west lines,pattern color=gray] plot [smooth cycle,tension=0.5] coordinates {(-1,0)(-0.9,.2)(-0.6,0.4)(-0.3,.4)(-0.1,0)(-0.3,-0.4)(-0.6,-0.4)(-0.9,-0.2)};

\node[below] at (0.7,-1.2) {$\{f\geq\ell\}$};
\draw (0.7,-1.2)--(-0.2,0);
\draw[fill] (-1,0) circle (1pt);
\draw (-1,0)--(-.5,-1.2);
\node[below] at (-0.5,-1.2) {$x$};

\draw[-{Triangle[width=18pt,length=8pt]}, line width=10pt](-1.2,0) -- (-2,0);
\draw[-{Triangle[width=18pt,length=8pt]}, line width=10pt](1.2,0) -- (2,0);

\begin{scope}[shift={(-3.2,0)}]
\draw[dashed] (-1,-1) rectangle (1,1);
\node[left] at (-1,0.6) {$B_1$};
\draw[thick,pattern=north west lines,pattern color=gray] plot [smooth cycle,tension=0.5] coordinates {(-0.9,0)(-0.9,.2)(-0.6,0.4)(-0.3,.4)(-0.1,0)(-0.3,-0.4)(-0.6,-0.4)(-0.9,-0.2)};
\draw[dashed] (-1,0.3) arc (90:-90:0.3);

\node[below] at (0,-1.2) {$\{f-\delta h\geq\ell\}$};
\draw (0,-1.2)--(-0.5,-0.3);
\node[above] at (-0.6,.6) {$W$};
\draw (-0.6,.6)--(-0.92,0.35);
\end{scope}

\begin{scope}[shift={(3.2,0)}]
\draw[dashed] (-1,-1) rectangle (1,1);
\node[left] at (-1,0.6) {$B_1$};
\begin{scope}
\clip (-1,-1) rectangle (1,1);
\draw[thick,pattern=north west lines,pattern color=gray] plot [smooth cycle,tension=0.5] coordinates {(-1.1,0)(-0.9,.2)(-0.6,0.4)(-0.3,.4)(-0.1,0)(-0.3,-0.4)(-0.6,-0.4)(-0.9,-0.2)};    
\end{scope}

\draw[dashed] (-1,0.3) arc (90:-90:0.3);
\node[below] at (0,-1.2) {$\{f-\delta h\geq\ell\}$};
\draw (0,-1.2)--(-0.5,-0.3);
\node[above] at (-0.6,.6) {$W$};
\draw (-0.6,.6)--(-0.92,0.35);
\end{scope}
\end{scope}

\end{tikzpicture}
\caption{Top panel: An illustration of a critical point which is `negatively pivotal' for both the excursion and level set component counts at level $\ell$. We consider a small perturbation $\pm\delta h$ of the excursion set on a neighbourhood $W$ of the critical point $x$, where $h$ is a local positive perturbation (precise definitions will be given in Section~\ref{s:vub}). In this case both excursion and level component counts decrease by one as we pass from the negative perturbation (left) to the positive perturbation (right). Middle panel: A critical point which is `negatively pivotal' for the level set and `not pivotal' for the excursion set. Bottom panel: A critical point of a boundary stratum which is `negatively pivotal' for both the level set and the excursion set. Note that other configurations are also possible.}
\label{fig:Pivotal1}
\end{figure}
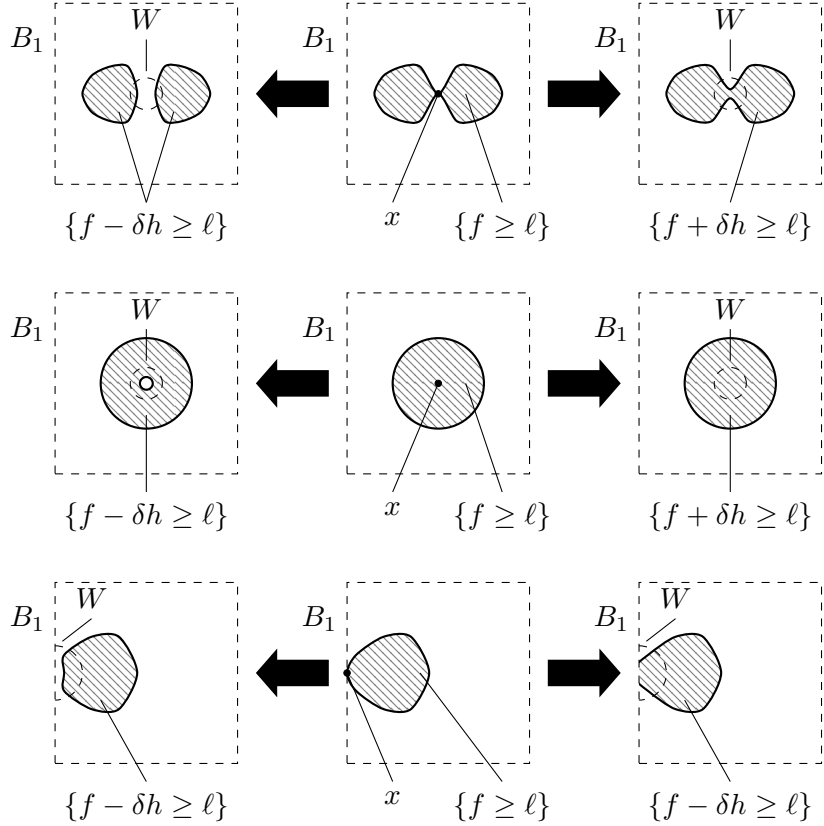

In Sections~\ref{ss:Pivotal}-\ref{ss:CriticalPoint} we will define the `two-point density function' $p_t(x,y)$, which can intuitively be thought of as the density of non-degenerate critical points of $(f,f^t)$ with critical value $(\ell,\ell)$. This function has a similar form to densities that appear in the Kac-Rice formula for the expected number of critical points of a field (see \cite[Chapter~11]{at07}) but there is no exact correspondence because the expected number of critical points of $(f,f^t)$ with critical value $(\ell,\ell)$ is zero. Let $p^{+ +}_t(x,y)$ be the fraction of this density which corresponds to pairs of critical points which are positively pivotal for $N_\star(B_1,\ell)$ and $N_\diamond(B_2,\ell)$ respectively, and define $p^{\pm \pm}_t(x,y)$ similarly. Although these densities depend on $B_1,B_2\subset\R^d$ and $\star,\diamond \in \{\mathrm{ES},\mathrm{LS}\}$ (as well as $f$ and $\ell$), to ease notation we drop these dependencies. We also note that, by definition,
\begin{equation}
\label{e:trivialub}
\sum_{i,j \in \{+,-\}} p^{ij}_t \le p_t .
\end{equation}

Strictly speaking, in the definition of $p_t(x,y)$ and  $p^{\pm \pm}_t(x,y)$ we also consider critical points of $(f,f^t)$ restricted to substrata of the boxes $B_1,B_2$ (i.e.\ boundary faces of dimension $0 \le m \le d-1$), so $p_t(x,y)$ and $p^{\pm \pm}_t(x,y)$ are actually measures which have absolutely continuous components with respect to (pairwise products of) the Lebesgue measure on all (pairs of) strata. (See Definition~\ref{d:PivotalMeasures} for a precise formulation of these statements.) Hence when we integrate $p^{\pm \pm}_t(x,y)$ over the domain $B_1\times B_2$, we will denote it as $d p^{\pm \pm}_t(x,y)$ to keep in mind that we integrate over Lebesgue measures on all strata.

\begin{theorem}[Covariance formula for component counts]
\label{t:cf}
Suppose Assumption \ref{a:gen} holds and let $\ell \in \R$, $B_1,B_2\subset\R^d $ be boxes, and $\star,\diamond \in \{\mathrm{ES},\mathrm{LS}\}$. Then
\begin{align*}  \Cov[N_\star(B_1),N_\diamond(B_2)]=  \int_0^1 \int_{B_1\times B_2}  \!  \! \! K(x-y) \Big( d  p_t^{++}(x,y) + d p_t^{--}(x,y) - dp_t^{+-}(x,y) -d  p_t^{-+}(x,y)  \Big)  \, dt .
\end{align*}
\end{theorem}

As mentioned, Theorem~\ref{t:cf} is inspired by an analogous covariance formula for `topological events' that appeared in \cite{bmr20}, and in fact we deduce Theorem~\ref{t:cf} as a consequence of that formula. In turn, these two formulae can be viewed as generalisations of classical interpolation formulae for finite Gaussian vectors \cite{pit96,bgh01}, and are also related to semigroup interpolation methods developed more recently \cite{led95,cha08,cha14}. See also \cite{tan15,mrv23,tv20} for other recent uses of similar interpolation formulae. For comparison let us state the interpolation formula in the classical case of smooth functions of finite Gaussian vectors.

\begin{proposition}[Classical Gaussian interpolation formula]
\label{p:gcf}
Let $X$ be a centred Gaussian vector with covariance matrix $(K(i,j))_{1 \le i,j \le n}$. Let $g_1,g_2 : \R^n \to \R$ be absolutely continuous functions such that $\E[g_i(X)^2]$ and  $\E[ \| \nabla g_i(X) \|_2^2 ]$ are finite for $i = 1,2$. Then, for every $t \in [0,1]$,
\[ \Cov[g_1(X), g_2(X^t)] = \int_0^t   \sum_{i,j} K(i,j) \E \Big[  \frac{\partial g_1(X)}{\partial X_i} \frac{\partial g_2(X^s)}{\partial X^s_j}   \Big] \, ds ,   \]
where
\[X^t = t X + \sqrt{1-t^2} \widetilde{X}   \ , \quad t \in [0,1] ,  \]
and $\widetilde{X}$ is an independent copy of $X$.
\end{proposition}
\begin{proof}
The proof in the case $g_1 = g_2$ can be found in \cite[Lemma 3.4]{cha08} (note the different parameterisation of the interpolation), and the proof in the general case is identical.
\end{proof}

Compared to the formula in Proposition \ref{p:gcf}, the covariance formula in Theorem~\ref{t:cf} has two main differences: (i) it concerns functions of a continuous Gaussian field instead of a finite Gaussian vector, and (ii) it concerns integer-valued (and so non-differentiable) functions. Heuristically one handles these differences by interpreting the partial derivatives in Proposition~\ref{p:gcf} as `densities' in an appropriate sense.

On the other hand, the covariance formula in Proposition \ref{p:gcf} does not require stationarity, whereas Theorem \ref{t:cf} is restricted to the stationary setting. We conjecture that Theorem~\ref{t:cf} should also be true for non-stationary Gaussian fields; indeed we only use stationarity in one place in the proof -- to obtain the integrability of the pivotal measures, see \eqref{e:qcbweak} -- and we expect this to be true under very general conditions.

\begin{remark}
We also conjecture that the covariance formula in Theorem~\ref{t:cf} should extend to component counts on $B_1$ and $B_2$  taken with respect to \textit{distinct} levels $\ell_1 \neq \ell_2$ respectively. In fact in that case we expect that the pivotal measures are uniformly bounded, so in theory it should be easier to verify the integrability in \eqref{e:qcbweak}. However this would require extra arguments and is not needed for our applications, so we do not consider it in this work.
\end{remark}

\subsection{Applications to variance asymptotics}
\label{s:app}

We next present applications of the covariance formula to the asymptotics of the variance of the component count. For brevity we write $N_\star(R)=N_\star(\Lambda_R)$, where $\Lambda_R=[-R/2,R/2]^d$.

\subsubsection{Known and conjectured results}
\label{s:known}
As previously mentioned, the second-order properties of the component count are expected to depend on the covariance structure of the field, of which there are three main categories:

\begin{enumerate}
\item We say that $f$ is \textit{short-range correlated} if $K \in L^1(\R^d)$; an important example is the \textit{Bargmann-Fock field} with $K(x) = e^{-|x|^2 / 2}$ (see \cite{bg17} for background and motivation). By analogy with known results for other `local' geometric statistics of the excursion/level sets (for instance their volume or Euler characteristic \cite{kl01,el16,an17,mul17,kv18a}), as well as with known results for the number of clusters in classical percolation \cite{cg84,zha01,pen01}, in this case it is expected that the variance of the component count $N_\star(R)$ has volume-order scaling $\asymp R^d$, and satisfies a central limit theorem (CLT).
\item We say that $f$ is \textit{regularly varying with index $\beta \in (0,d)$ and remainder $L$} if $K(x) = |x|^{-\beta} L(|x|)$ where $L$ is a slowly varying function (see \cite{bgt87} for a definition); important examples are the \textit{Cauchy fields} with covariance $K(x) = (1 + |x|^2)^{-\beta/2}$. The expected second-order properties are not clear in this case, although by analogy with the zero count for Gaussian processes \cite{slu94} one might suppose that the variance has super-volume scaling $\asymp R^{2d-\beta} L(R)$ for most levels, with the possibility of a finite number of \textit{anomalous levels} with lower-order variance (although this latter phenomena is not well-understood, see \cite[Section 2.2]{bmm22}). It is unclear whether a CLT holds in this case.
\item We say that $f$ has \textit{oscillating long-range correlations} if it is neither short-range correlated nor positively correlated. An important example is the \textit{monochromatic random wave} for which $K(x) = |x|^{-(d/2-1)} J_{d/2-1}(|x|)$ where $J_n$ is the order-$n$ Bessel function; equivalently $K$ is the Fourier transform of the normalised Lebesgue measure on the sphere $\mathbb{S}^{d-1}$. In $d=2$ this is the \textit{random plane wave} (RPW) with $K(x) = J_0(|x|)$. Again by analogy with known results for `local' statistics \cite{mrw20,npr19}, one might expect that for the monochromatic random wave the variance has super-volume scaling $\asymp R^{d+1}$ for most levels, with the possibility of anomalous levels with lower-order variance. Again it is unclear whether a CLT holds in this case.
\end{enumerate}

The predicted behaviour described above has so far only been fully verified for a restricted class of short-range correlated fields. In particular, it was recently shown in \cite{bmm23} that, for a certain subclass of fields whose correlations decay as
\begin{displaymath}
| K(x) | \leq c(1+\lvert x\rvert)^{-\beta}
\end{displaymath}
for some $\beta>9d$, $N_\star(R)$ satisfies a CLT with volume-order variance scaling. Although this subclass contains important examples such as the Bargmann-Fock field, the assumption on the decay of $K$ is quite strong, and it appears the method in \cite{bmm23} cannot be extended, without significant new ideas, to cover the entire short-range correlated case (see \cite[Remark~3.12]{bmm23}).

Outside of the short-range correlated fields covered by the result in \cite{bmm23}, the existing second-order results are limited to variance bounds which are mostly sub-optimal and apply only to planar fields. We now briefly describe these, starting with lower bounds.

In \cite{ns20} Nazarov and Sodin proved a polynomial lower bound
\[ \mathrm{Var}[N_\star(R)] \geq c R^\eta \]
valid for general planar fields with polynomially decaying correlations (to be more precise, they considered families $(f_n)_{n\geq 1}$ of Gaussian fields defined on the sphere $\mathbb{S}^2$ which converge locally, and only considered the nodal set, but we expect the proof extends, up to boundary effects, to the Euclidean setting we consider here). The exponent $\eta > 0$ was not quantified but is small. In \cite{bmm22} sharper results were proven for planar fields under stronger conditions. More precisely, if $f$ is short-range correlated, and if $\int K(x) dx > 0$ and $\frac{d}{d\ell}\mu_\star(\ell)\neq 0$ where $\mu_\star(\ell)$ is the asymptotic mean number of components per unit volume at level $\ell$, then
\[ \mathrm{Var}[N_\star(R)]\geq cR^2 . \]
Further, in \cite{bmm20a} the condition $\frac{d}{d\ell}\mu_\star(\ell)\neq 0$ was shown to hold for a large range of levels (including the zero level for excursion sets). It was also shown in \cite{bmm22} that if $f$ is regularly varying with index $\beta \in (0,2)$ and remainder $L$ (to be more precise, the condition in \cite{bmm22} was given in terms of the spectral measure, but it is roughly equivalent) then, under the same derivative condition,
\[ \mathrm{Var}[N_\star(R)]\geq cR^{2-\beta} L(R) , \]
and if $f$ is the RPW then, under the same derivative condition but excluding level $\ell = 0$,
\[ \mathrm{Var}[N_\star(R)]\geq cR^3 . \]
Note that all these bounds are expected to be of optimal order whenever they apply.

Turning to upper bounds, it is straightforward to establish that, in all dimensions,
\begin{equation}
	\label{e:ub}
	\mathrm{Var}[N_\star(R)] \le c R^{2d} 
\end{equation}
 using a comparison with critical points. More precisely, since each excursion (resp.\ level) set component contains (resp.\ surrounds) at least one critical point, the component count in a compact domain is bounded by the number of critical points in the domain. Since the latter quantity has a second moment of order $\asymp R^{2d}$ \cite{eli85,ef16}, we deduce \eqref{e:ub}. Note that this bound is only expected to be attained for very degenerate Gaussian fields (see \cite{bmm20,bmm22} for examples). For the nodal level $\ell = 0$ of the random plane wave, Priya \cite{pri20} recently improved the upper bound to $c R^{4 - 1/8}$. This bound is deduced from the exponential concentration of the nodal component count, and the proof does not extend to general levels/fields. While weaker concentration bounds have also been established for $N_\star(R)$ in general \cite{rv19, bmr20}, these do not lead to improved bounds on the variance.

\subsubsection{A formula for the variance: monotonicity and convexity}

We first specialise the covariance formula to the case of the variance, which enjoys additional monotonicity and convexity properties:

\begin{theorem}[Variance formula for the component count]
\label{t:vf}
Suppose Assumption \ref{a:gen} holds and $\ell \in \R$, $\star \in \{\mathrm{ES},\mathrm{LS}\}$, and $R > 0$. Then the map
\begin{equation}
\label{e:covmap}
 t \mapsto \E \big[ N_\star(R)    N^t_\star(R) \big]    \ ,  \quad t \in [0,1] , 
 \end{equation}
is non-decreasing, convex, and absolutely continuous with derivative (almost everywhere)
\begin{equation}
\label{e:der}
     \int_{\Lambda_R \times \Lambda_R}  \! \! \! K(x-y) \Big( d p_t^{++}(x,y) + d p_t^{--}(x,y) - d p_t^{+-}(x,y) - d p_t^{-+}(x,y)  \Big)  ,
  \end{equation}
  where $dp^{\pm \pm}_t$ are the `pivotal measures' as in Theorem \ref{t:cf} in the case $B_1 = B_2 = \Lambda_R$ and $\star = \diamond$. In particular,
\begin{equation}
\label{e:vf}
\Var[N_\star(R)]=\int_0^1 \int_{\Lambda_R \times \Lambda_R}  \!  \! \! K(x-y) \Big( d  p_t^{++}(x,y) + d p_t^{--}(x,y) - dp_t^{+-}(x,y) -d  p_t^{-+}(x,y)  \Big)  \, dt .
\end{equation}
\end{theorem}

\begin{remark}
\label{r:cd}
It seems plausible that, in fact, $t \mapsto \E \big[ N_\star(R)N^t_\star(R) \big]$ is continuously differentiable under the conditions of Theorem~\ref{t:vf}. One possible strategy to prove this statement would be to show that the pivotal intensities $p_t^{\pm \pm}(x,y)$ are continuous functions of $(t,x,y) \in [0,1) \times B_1 \times B_2$ when restricted to pairs of substrata. We do not pursue such an argument here as it would not directly impact on our applications.
\end{remark}

\subsubsection{Exact asymptotics in the short-range correlated case}

In the case $K \in L^1(\R^d)$, we can deduce from Theorem \ref{t:vf} the asymptotic behaviour of the variance:

\begin{theorem}\label{t:Asymptotic}
Suppose Assumption \ref{a:gen} holds and let $\star \in \{\mathrm{ES}, \mathrm{LS}\}$ and $\ell \in \R$. If $f$ is short-range correlated (i.e.\ $K\in L^1(\R^d)$) then as $R\to\infty$
\begin{equation}
    \label{e:asymptotic}
\mathrm{Var}[N_\star(R)]= \sigma^2 R^d+o(R^d) ,
\end{equation}
where $\sigma^2 \in [0, \infty)$ is the constant defined as
\begin{equation}
\label{e:sigma}
 \sigma^2 =\int_0^1  \int_{\R^d}K(x)\Big(\hat{I}_t^{++}(0,x) + \hat{I}_t^{--}(0,x) - \hat{I}_t^{+-}(0,x) -\hat{I}_t^{-+}(0,x)  \Big)  \,dx\, dt
\end{equation}
for the `stationary pivotal intensities' $\hat{I}_t^{\pm\pm}$ defined in Section~\ref{ss:Asymptotic}. 
\end{theorem}

As mentioned in Section~\ref{s:known} above, the asymptotic \eqref{e:asymptotic} was already known for a subset of short-range correlated fields, including the Bargmann-Fock field \cite{bmm23}, although using a very different `martingale' approach. This approach led to an alternative representation for $\sigma^2$ (see \cite[Eq. (3.34)]{bmm23}) which is interesting to compare to \eqref{e:sigma}. However \eqref{e:asymptotic} is new for all short-range fields whose correlations do not decay faster than $|x|^{-9d}$, and the expression \eqref{e:sigma} is new in all cases.

Roughly speaking we deduce Theorem \ref{t:Asymptotic}  from Theorem~\ref{t:vf} by showing that the pivotal measures $dp_t^{\pm\pm}(x,y)$ can be replaced with their stationary counterparts $\hat{I}_t^{\pm \pm}(x,y) dx dy$. More precisely, this involves showing that (i) the contribution from the lower dimensional strata of $\Lambda_R$ are bounded by $O(R^{d-1})$, and (ii) the dependence of the pivotal measures on the domain $\Lambda_R$ is negligible for the bulk of points $x,y\in\Lambda_R$.

The previous result determines the asymptotic growth rate of the variance whenever $\sigma^2 > 0$. A natural question is whether this is always true:

\begin{question}
Is the constant $\sigma^2 \ge 0$ in Theorem \ref{t:Asymptotic} strictly positive in general?
\end{question}

At present we only know how to verify $\sigma^2 > 0$ for a restricted class of fields \cite{bmm22,bmm23}, which includes the Bargmann-Fock field, although we expect it to be true in general.

As mentioned in Remark~\ref{r:cd} above, it may be possible to show that $t \mapsto \E \big[ N_\star(R)    N^t_\star(R) \big]   $ is continuously differentiable. Assuming this, using the convexity in Theorem \ref{t:vf} and an argument similar to that in Theorem~\ref{t:Asymptotic} one would obtain that 
\begin{align}
    \label{e:Asymptotic2}
  \sigma^2  & \ge 
\int_{\R^d} K(x) \Big(\hat{I}_0^{++}(0,x) + \hat{I}_0^{--}(0,x) - \hat{I}_0^{+-}(0,x) -\hat{I}_0^{-+}(0,x)  \Big)    \, dx   \\
 & \nonumber = \big(\hat{I}^+ - \hat{I}^-\big)^2 \int_{\R^d}K(x) \, dx 
  \end{align}
where $\hat{I}^{+} = \hat{I}^{+ +}_0(0,0)^{1/2}$ and $\hat{I}^{-} = \hat{I}^{--}_0(0,0)^{1/2}$ can be interpreted as `one-point stationary pivotal intensities', and in the last step we used the fact that $f$ and $f^0 = \tilde{f}$ are independent.

Such analysis, if carried out, would show that $\sigma^2 > 0$ as long as
\begin{equation}
    \label{e:criteria}
  \int_{\R^d}K(x) \, dx > 0  \qquad \text{and} \qquad   \hat{I}^+ \neq \hat{I}^-.
  \end{equation} 
It is interesting to compare \eqref{e:criteria} to the criteria for $\sigma^2 > 0$ previously established in \cite{bmm22,bmm23}:
\begin{itemize}
    \item The condition $ \int_{\R^d}K(x) \, dx > 0$ was also required in \cite{bmm22,bmm23}, and is quite natural -- it implies that the reproducing kernel Hilbert space of the field contains functions which are approximately constant -- although we do not believe it to be necessary for $\sigma^2 > 0$.
  \item The condition $\hat{I}^+ \neq \hat{I}^-$ was also (implicitly) present in \cite{bmm22}, although was not needed in \cite{bmm23}. In fact, we previously showed in \cite{bmm22} that, for a wide class of planar fields,
\[ \frac{d \mu_\star(\ell)}{d\ell} =  \hat{I}^+ - \hat{I}^-      , \]
so that the condition $\hat{I}^+ \neq \hat{I}^-$ is equivalent to the condition $ \frac{d}{d\ell} \mu_\star(\ell) \neq 0$ mentioned in Section~\ref{s:known} above. Hence if one were to make the above arguments rigorous, it would recover the bound $\mathrm{Var}[N_\star(R)] \ge c R^2$ proven in \cite{bmm22} for planar fields, using a very different method.
\end{itemize}

\subsubsection{General variance bounds}

Our next application of Theorem~\ref{t:vf} is to deduce a general upper bound on the variance that applies outside the short-range correlated case. By bounding the pivotal intensities with the critical point intensities (see \eqref{e:trivialub}) we immediately obtain the following:

\begin{corollary}
\label{c:vf}
Suppose Assumption \ref{a:gen} holds and let $\ell \in \R$, $\star \in \{\mathrm{ES},\mathrm{LS}\}$, and $R > 0$. Then
\begin{equation}
\label{e:cvfub1}
 \Var[N_\star(R)]   \le   \int_0^1  \int_{\Lambda_R \times \Lambda_R}  \! \! |K(x-y)|  d p_t(x,y)   \, dt  . 
 \end{equation}
\end{corollary}

In Section \ref{s:dbcp} we prove that the critical point intensities $p_t(x,y)$ are integrable on the diagonal (i.e.\ as $t \to 1$ and $x-y \to 0$), from which we deduce the following upper bound:

\begin{theorem}[Variance upper bound]
\label{t:vub}
Suppose Assumption \ref{a:gen} holds and let $\ell \in \R$ and $\star \in \{\mathrm{ES}, \mathrm{LS}\}$. Then there exists a $c > 0$ such that, for all $R \ge 1$,
\begin{equation}
\label{e:vub}
 \Var[ N_\star(R) ] \le c R^d  \int_{\Lambda_{2R}}  \! \! \widetilde{K}(x) \, dx    , 
 \end{equation}
 where $\widetilde{K}(x) = \sup_{ y : |y-x| \le 1} |K(y)|$.
\end{theorem}

\begin{remark}
The proof of Theorem~\ref{t:vub} shows that the constant $c$ could be replaced by ${c'(1 + \ell^{2d})}$ for a $c' > 0$ depending only on the field $f$. The term $ R^d \int_{\Lambda_{2R}}  \! \! \widetilde{K}(x) \, dx$ could also be replaced by $\int_{\Lambda_R \times \Lambda_R} |K(x-y)| dx dy$, where this double integral is understood as being over the box $\Lambda_R$ as well as its boundary faces of all dimensions equipped with their respective Lebesgue measures.
\end{remark}

To illustrate the optimality (or lack thereof) of this result, let us consider some examples:

\begin{corollary}
\label{c:vf2}
Suppose Assumption \ref{a:gen} holds and let $\ell \in \R$ and $\star \in \{\mathrm{ES}, \mathrm{LS}\}$.
\begin{enumerate}
\item If $f$ is regularly varying with index $\beta \in (0, d)$ and remainder $L$, there exists a $c > 0$ such that, for all $R \ge 1$,
\[ \Var[ N_\star(R) ] \le c R^{2d-\beta} L(R) .  \]
\item If $f$ is the monochromatic random wave, there exists a $c > 0$ such that, for all $R \ge 1$,
\[ \Var[ N_\star(R) ] \le c R^{(3d+1)/2} .  \]
\end{enumerate}
\end{corollary}

\begin{proof}
In the regularly varying case this follows directly from Theorem \ref{t:vub} and the additional fact that, as $R \to \infty$,
\[  \int_{\Lambda_{2R}} \widetilde{K}(x) dx \le \int_{|x| \le \sqrt{d}R} \widetilde{K}(x) dx \sim \int_{|x| \le \sqrt{d}R} K(x) dx  \sim \frac{1}{d-\beta} (\sqrt{d}R)^{d-\beta}L(R)  \]
by regular variation and Karamata's integral theorem \cite[Theorem 1.5.11]{bgt87}. In the monochromatic case we additionally use that, for $R \ge 1$,
\[  \int_{\Lambda_{2R}} \widetilde{K}(x) dx \le \int_{|x| \le \sqrt{d}R} \widetilde{K}(x) dx = O( R^{(d+1)/2} ) , \]
since $K(x) = |x|^{-(d/2-1)} J_{d/2-1}(|x|) = O( |x|^{-(d-1)/2} )$.
\end{proof}

Comparing with the discussion in the previous section, the bound in Theorem \ref{t:vub} is (conjecturally) of the correct order in the regularly varying case, but not necessarily in the oscillating case due to the fact that we have $\widetilde{K} \approx |K|$ instead of $K$ in the right-hand side of~\eqref{e:vub}. Nevertheless, even in the oscillating case the attained bounds are best known. For instance, for the nodal set of the RPW the attained bound $c R^{7/2}$ improves the previously best-known bound $cR^{4-1/8}$ \cite{pri20} (and moreover is valid at all levels). 

\begin{remark}
\label{r:gub}
As in the discussion around \eqref{e:Asymptotic2}, if one knew the continuous differentiability of $t \mapsto \E \big[ N_\star(R)    N^t_\star(R) \big]   $ (see Remark \ref{r:cd}) one could use the convexity in Theorem~\ref{t:vf} to obtain the alternative upper bound
\begin{equation}\label{e:cvfub2}
\begin{aligned}
 \Var[N_\star(R)]  & \le \lim_{t \to 1} \int_{\Lambda_R \times \Lambda_R} \! \! \! \! K(x-y) \Big( d p_t^{++}(x,y) + d p_t^{--}(x,y) - d p_t^{+-}(x,y) - d p_t^{-+}(x,y)  \Big)  , \\
   & \le  \lim_{t \to 1} \int_{\Lambda_R \times \Lambda_R} \! \! \! \! |K(x-y)|  d p_t(x,y)  .
\end{aligned}
\end{equation}
This is comparable to the \textit{Gaussian Poincar\'{e} inequality}
\[   \Var[ f(X) ] \le \sum_{i,j} K(i,j) \E \Big[ \frac{\partial f(X)}{\partial X_i} \frac{\partial f(X)}{\partial X_j} \Big] \le  \sum_{i,j} |K(i,j)| \Big| \E \Big[ \frac{\partial f(X)}{\partial X_i} \frac{\partial f(X)}{\partial X_j} \Big] \Big|   \]
for Gaussian vectors $X$ and sufficiently smooth $f$ (see, e.g., \cite[Corollary~3.2]{che82}). However it is not clear whether the bound \eqref{e:cvfub2} is effective  in our context, since our analysis in Section~\ref{s:dbcp} does \textit{not} show that $ \int_{\Lambda_R \times \Lambda_R}  d p_t(x,y) $ remains bounded as $t \to 1$ (c.f.\ \eqref{e:cvfub1} where we only require the integrability of $\int_{\Lambda_R \times \Lambda_R}  d p_t(x,y)$ as $t \to 1$).
\end{remark}

\begin{remark}[General lower bounds]
As in the previous remark (see also the discussion around \eqref{e:Asymptotic2}), if one knew the continuous differentiability of $t \mapsto \E \big[ N_\star(R)    N^t_\star(R) \big]   $ one could use the convexity in Theorem \ref{t:vf} to obtain the general lower bound 
\begin{align}
\label{e:cvflb}
  \Var[N_\star(R)] & \ge   \int_{\Lambda_R \times \Lambda_R}  \! \! \! K(x-y) \Big( d p_0^{++}(x,y) + d p_0^{--}(x,y) - d p_0^{+-}(x,y) - d p_0^{-+}(x,y)  \Big) \, dx dy
 \\
  & \nonumber =   \int_{\Lambda_R \times \Lambda_R}  \! \! K(x-y) \big( d p^+(x) - d p^-(x) \big) \big(  d p^+(y) - d p^-(y) \big)  \, dx dy , 
  \end{align}
  where $dp^{\pm}$ are `one-point pivotal measures' defined similarly to $\hat{I}^{\pm}$ (we refrain from giving a formal definition), and the last step uses the independence of $f$ and $f^0 = \tilde{f}$. This is similar in spirit to the bound in \cite[Proposition~3.7]{cac82} for smooth functions of finite Gaussian vectors, which is proven using a very different method.
\end{remark}

\subsection{Outline of the paper}
In Section~\ref{s:vub} we prove the covariance and variance formulae in Theorems~\ref{t:cf} and~\ref{t:vf}, and also deduce Theorems~\ref{t:Asymptotic} and~\ref{t:vub} as a consequence, subject to some auxiliary results, namely the integrability of the pivotal measures and the topological stability of the component count. The integrability of the pivotal measures is proven in Section~\ref{s:dbcp}, and the topological stability is proven in Section \ref{s:top}. Finally in the appendix we collect some non-degeneracy properties of smooth Gaussian fields.

\subsection{Acknowledgements}
The authors are grateful to Hugo Vanneuville for pointing out in \cite{van19} that interpolation methods can be used to study the variance of the component count. We also thank an anonymous referee for detailed references to results in \cite{gm88}, as well as general comments and corrections which improved the presentation of this work.

The second author was supported by the European Research Council (ERC) Advanced Grant QFPROBA (grant number 741487). The third author was supported by the Australian Research Council (ARC) Discovery Early Career Researcher Award DE200101467, and acknowledges the hospitality of the Statistical Laboratory, University of Cambridge, where part of this work was carried out.


\section{An exact formula for the covariance and applications}
\label{s:vub}

In this section we establish the exact formula for the component count covariance stated in Theorem~\ref{t:cf} above, and then deduce the various consequences of this formula. The proof makes use of auxiliary results (integrability of the pivotal measures and topological stability of the component count) which are proven in the following sections.

\subsection{Stability of the component count}
\label{s:stability}
We begin by stating three deterministic topological stability properties of the component count. Roughly speaking the first states that the component count at level $\ell$ cannot change under small perturbations if there are no critical points at level $\ell$. The second states the component count can change by at most one if the function is perturbed near a non-degenerate critical point. The third states that the change in the component count when perturbing near a critical point is the same for all sufficiently large boxes. These properties are all consequences of standard results of stratified Morse theory (see \cite{gm88} for a comprehensive background) as we will explain below. However in the interest of completeness, we present proofs of our results in Section~\ref{s:top}.

To formalise the stability properties we require the notion of stratified sets. A box $B=[a_1,b_1]\times\dots\times[a_d,b_d]$ may be considered as a \emph{stratified set} by partitioning it into the finite collection $\mathcal{F} = (F_i)$ of (the interiors of) its faces of all dimensions $0 \le m \le d$, which we refer to as the \textit{strata}. We let $\mathcal{F}_0 \subset \mathcal{F}$ be the strata of dimension $m = 0$. A \textit{stratified box} will be any such $B$ considered as a stratified set. We equip each stratum $F \in \mathcal{F} \setminus \mathcal{F}_0$ with its Lebesgue measure $dv_F$, and each stratum $F \in \mathcal{F}_0$ with a Delta mass, also denoted~$dv_F$. 

Let $B$ be a stratified box and let $U$ be an open neighbourhood of $B$. For $x \in B$ and $g \in C^1(U)$, $\nabla_F g(x)$ denotes the derivative of $g$ restricted to the unique stratum $F$ containing $x$. A \textit{(stratified) critical point} of $g$ is a point $x \in F$ such that $\nabla_F g(x) = 0$. The \textit{level} of this critical point is the value $g(x)$. The critical point is \textit{non-degenerate} provided that $\det\nabla_F^2 g(x) \neq 0$ and if $x\in \bar{F'}$ for a higher dimensional strata $F'$, then $\nabla_{F'} g(y) \not \to 0$ as $y\to x$, $y\in F'$. We take any point belonging to a stratum of dimension $m=0$ to be a non-degenerate critical point by definition.
 
Extending our previous notation, for a function $g \in C^1(U)$ we let $N_\star(B,g,\ell)$ denote the component count of $g$ at level $\ell$ inside $B$, i.e.\ the number of connected components of $\{g \ge \ell\}$ (if $\star=\mathrm{ES}$) or $\{g = \ell\}$ (if $\star = \mathrm{LS}$) which intersect $B$ but not its boundary. 

The first two stability properties are the following:

\begin{lemma}[Stability away from critical points]\label{l:Topological_stability}
Suppose that $g\in C^1(U)$ has no (stratified) critical points at level $\ell\in\R$. There exists a neighbourhood of $g$ in $C^1(U)$ such that for all $g^\prime$ in the neighbourhood
\begin{displaymath}
N_\star(B,g^\prime,\ell)= N_\star(B,g,\ell).
\end{displaymath}

\end{lemma}

\begin{lemma}[Change near critical points]
\label{l:change}
Suppose $g \in C^2(U)$ has a non-degenerate stratified critical point at $x$ at level~$\ell$. Then there exists a unique pair $(k, k') \in \N_0 \times \N_0$, with $|k-k'| \le 1$, satisfying the following: for any sufficiently small open neighbourhood $W$ of $x$, any function $h \in C_c^2(W)$ such that $h(x)>0$, and every sufficiently small $\delta > 0$,
\[ N_\star(B, g - \delta h, \ell) = k \quad \text{and} \quad  N_\star(B, g + \delta h, \ell) = k' . \]
\end{lemma}

Lemma~\ref{l:Topological_stability} follows quite directly from Thom's first isotopy lemma, as will become apparent from our proof in Section~\ref{s:top}. More general forms of this result are proven in Part I, Chapter 4 and Chapter 7.4 of \cite{gm88}. Lemma~\ref{l:change} follows from the very general `Main Theorem' (Theorem~3.5.4) of \cite{gm88}, although again we will give a proof in our particular setting in Section~\ref{s:top}.

In particular, the second stability property allows us to classify all non-degenerate critical points depending on whether the component count increases by one, decreases by one, or stays the same, after perturbation. More precisely, define the index sets
\[  P^+ =   \{ (k, k') \in \N_0 \times \N_0 : k' - k = 1 \} \ , \quad   P^- =   \{ (k, k') \in \N_0 \times \N_0 : k' - k = -1 \}  , \]
\[ P = P^+ \cup P^- \quad \text{and} \quad P^0 = \{ (k, k') \in \N_0 \times \N_0 : k' - k = 0 \}. \]  
We say that a non-degenerate critical point is \textit{positively pivotal} (for the component count) if $(k,k') \in P^+$, \textit{negatively pivotal} if $(k,k') \in P^-$, and \textit{not pivotal} if $(k,k') \in P^0$, where $(k,k')$ is the pair guaranteed to exist by Lemma \ref{l:change}. Note that these definitions depend on the box $B$. When we wish to emphasise this we will say that the critical point is $(\Delta,B)$-pivotal for $\Delta=+,-,0$ respectively.

Finally we state our third stability property, which also follows as a consequence of the `Main Theorem' of \cite{gm88}:

\begin{lemma}[Asymptotic topological stabilisation]\label{l:AsympStab}
Suppose $g \in C^2(\R^d)$ has a non-degenerate critical point at $x$ at level $\ell$. Then there exists $r>0$ and $\Delta \in\{+,-,0\}$ such that, for all boxes $B$ containing $x+\Lambda_r$, the critical point $x$ is $(\Delta,B)$-pivotal.
\end{lemma}

\subsection{Proof of the covariance formula}
\label{s:vfproof}

In this subsection we prove the covariance formula in Theorem \ref{t:cf}. Throughout the rest of this section we fix $\ell \in \R$, boxes $B_1,B_2\subset\R^d$, and  $\star,\diamond \in \{\mathrm{ES},\mathrm{LS}\}$, and drop these variables from the notation. For brevity we write $N = N_\star(B_1)$ and $N^t = N^t_\diamond(B_2)$, recalling that $N^1 = N^1_\diamond(B_2) = N_\diamond(B_2)$.

\subsubsection{Gaussian conditioning}\label{ss:GaussianRegression}
Throughout this paper we will make use of conditioning for Gaussian vectors; that is, given two jointly Gaussian vectors $X$ and $Y$ of length $m_1$ and $m_2$ respectively we will consider variables of the form $(Y\;|\;X=x)$ for some $x\in\R^{m_1}$. Since $\{X=x\}$ may have measure zero, to make sense of this definition we condition on the event $\lvert X-x\rvert<\epsilon$ and take $\epsilon\downarrow 0$. This is sometimes known as `vertical window' conditioning (and is the most common form of conditioning on a continuous variable taking a particular value). This form of conditioning is always well-defined whenever $X$ is non-degenerate. To see this, one can define the variable $Z$ by
\begin{displaymath}
    Y=Z+\Sigma_{YX}\Sigma_X^{-1}X
\end{displaymath}
where $\Sigma_{YX}=\Cov[Y,X]$ and $\Sigma_X=\Cov[X]$. Basic linear algebra shows that $Z$ is uncorrelated with $X$ and hence, by Gaussianity, independent. Then by the above equation $(Y\;|\;X=x)=Z+\Sigma_{YX}\Sigma_X^{-1}x$. From this we see that the conditioned variable is also Gaussian and we can explicitly compute its mean and covariance matrix as
\begin{displaymath}
    \E[Y\;|\;X=x]=\E[Y]+\Sigma_{YX}\Sigma_X^{-1}(x-\E[X])\qquad\text{and}\qquad\Cov[Y\;|\;X=x]=\Cov[Y]-\Sigma_{YX}\Sigma_X^{-1}\Sigma_{YX}^T.
\end{displaymath}
This relationship between the conditional and unconditional moments is sometimes known as \emph{Gaussian regression}.

\subsubsection{The pivotal measures}\label{ss:Pivotal}
We begin by defining the pivotal measures $dp^{\pm \pm}_t(x,y)$ that appear in Theorem \ref{t:cf}.

Throughout this section we view $B_1$ and $B_2$ as stratified boxes in the sense of Section~\ref{s:stability} and fix open neighbourhoods $U_1,U_2 \subset \R^d$ of $B_1$ and $B_2$ respectively. We denote by $\mathcal{F}^{(1)},\mathcal{F}^{(2)}$ the set of strata of $B_1,B_2$ respectively, $\mathcal{F}^{(i)}_0 \subset \mathcal{F}^{(i)}$ the subset of strata of dimension $m=0$, and write $\mathcal{F}^{(1,2)}=\mathcal{F}^{(1)}\times\mathcal{F}^{(2)}$. Recall that each stratum $F \in \mathcal{F}^{(i)} \setminus \mathcal{F}^{(i)}_0$ is equipped with its Lebesgue measure $dv_F$, and each stratum $F \in \mathcal{F}^{(i)}_0$ with a delta mass, also denoted $dv_F$. 

For $x\in B_1$ we define $F_x=F_x(B_1)$ to be the unique stratum of $B_1$ containing $x$ and for $y\in B_2$ we define $F_y=F_y(B_2)$ similarly. Let us fix momentarily $t \in (0, 1]$, $x\in B_1$ and $y\in B_2$. Consider the fields $f$ and $f^t$ conditioned on
\begin{equation}
    \label{e:cond}
(f(x),f^t(y),\nabla_{F_x} f(x), \nabla_{F_y} f^t(y)) = (\ell,\ell,0,0).
\end{equation} 
This conditioning is well-defined since the vector on the left-hand side of \eqref{e:cond} is non-degenerate even if $x=y$ (see Lemma \ref{l:nondegen2} for a proof).

For a square matrix $M$, we write $\lvert M\rvert$ to denote the absolute value of the determinant of $M$. Using the fact that $(f(x),\nabla f(x),f(y),\nabla f(y))$ is a non-degenerate Gaussian vector (see Remark~\ref{r:nondegen}) we see that under the conditioning \eqref{e:cond} the partial derivatives of $f$ at $x$ in directions normal to $F_x$ must be non-zero (and similarly for $f^t$ at $y$). Therefore, assuming that $|\nabla_{F_x}^2 f(x)|$ and $|\nabla_{F_y}^2 f^t(y) |$ are non-zero, we let $A(\iota_1,\iota_2)$ denote the event that the pairs $(k, k')$, guaranteed to exist by Lemma \ref{l:change} for $f$ and $f^t$, are given by $\iota_1 , \iota_2 \in P^+ \cup P^- \cup P^0$ respectively. In particular we have the decomposition
\begin{equation}
\label{e:partition}
    |\nabla_{F_x}^2 f(x) \nabla_{F_y}^2 f^t(y) | =  |\nabla_{F_x}^2 f(x) \nabla_{F_y}^2 f^t(y) | \sum_{ \iota_1,\iota_2 \in P^+ \cup P^- \cup P^0  }\id_{A(\iota_1,\iota_2)} .
\end{equation}

\begin{definition}\label{d:PivotalMeasures}
For each $\iota_1, \iota_2 \in P^+ \cup P^- \cup P^0$ we define the associated \textit{two-point intensity}
\[ I^{\iota_1, \iota_2}_t(x,y) =  \varphi_t(x,y) \widetilde{\E} \big[ |\nabla_{F_x}^2 f(x) \nabla_{F_y}^2 f^t(y) | \id_{A(\iota_1,\iota_2)}  \big] , \]
where $\varphi_t(x,y)$ denotes the density of the conditioning \eqref{e:cond} (i.e.\ the density of the non-degenerate Gaussian vector $(f(x),f^t(y),\nabla_{F_x} f(x), \nabla_{F_y} f^t(y))$ evaluated at $(\ell,\ell,0,0)$), and $\widetilde{\E}$ denotes expectation under this conditioning. If $F_x \in \mathcal{F}^{(1)}_0$ (the set of strata with dimension $0$), then by convention we remove $\nabla_{F_x} f$ from the conditioning \eqref{e:cond} and corresponding density $\varphi_t(x,y)$ and set $|\nabla^2_{F_x} f(x) | = 1$, and similarly for $F_y\in \mathcal{F}^{(2)}_0$.

We also define a corresponding measure equal to the product of Lebesgue measures on each pair of strata weighted by this intensity:
\[ dp^{\iota_1, \iota_2}_t(x,y) =I^{\iota_1,\iota_2}_t(x,y) dv_{F_x}(x) dv_{F_y}(y) . \]

The \textit{pivotal intensities} $I^{\pm,\pm}_t(x,y)$, and their corresponding \textit{pivotal measures} $dp^{\pm,\pm}_t(x,y)$, are then defined as
\begin{equation}
\label{e:pivint}
I^{\pm,\pm}_t(x,y) = \sum_{\iota_1 \in P^\pm, \iota_2 \in P^\pm } I^{\iota_1,\iota_2}_t(x,y) \;\; \text{and} \;\; dp^{\pm, \pm}_t(x,y) =  I^{\pm,\pm}_t(x,y) dv_{F_x}(x) dv_{F_y}(y) .
\end{equation}
\end{definition}

Let us observe that 
\begin{equation}
    \label{e:sumint1}
 I^{+,+}_t(x,y) + I^{-,-}_t(x,y) + I^{+,-}(x,y) + I^{-,+}_t(x,y)  =  \sum_{\iota_1 \in P, \iota_2 \in P }  I^{\iota_1,\iota_2}_t(x,y),
 \end{equation}
and
\begin{equation}
    \label{e:sumint2}
 I^{+,+}_t(x,y) + I^{-,-}_t(x,y) - I^{+,-}(x,y) - I^{-,+}_t(x,y)  =  \sum_{\iota_1 \in P, \iota_2 \in P }  \sigma(\iota_1)\sigma(\iota_2) I^{\iota_1,\iota_2}_t(x,y),
\end{equation}
where $\sigma(\iota) = \pm 1$ if  $\iota \in P^\pm$.

\subsubsection{The critical point measure}\label{ss:CriticalPoint}

We next define the \textit{critical point intensity} and \textit{measure} as
\begin{equation}\label{e:CritInt}
I_t(x,y) = \varphi_t(x,y) \widetilde{\E} \big[ |\nabla_{F_x}^2 f(x) \nabla_{F_y}^2 f^t(y) | \big]
\end{equation}
and
\[ dp_t(x,y) = I_t(x,y) dv_{F_x}(x) dv_{F_y}(y) \]
respectively. Note that \eqref{e:partition}--\eqref{e:sumint1} imply the bound (c.f.\ \eqref{e:trivialub})
\begin{equation}\label{e:intbound}
I_t(x,y) = \sum_{\iota_1, \iota_2 \in P \cup P^0}  I^{\iota_1,\iota_2}_t(x,y)   \ge \sum_{\iota_1, \iota_2 \in P}  I^{\iota_1,\iota_2}_t(x,y) = I^{+,+}_t(x,y) +   I^{-,-}_t(x,y) +   I^{+,-}_t(x,y) +  I^{-,+}_t(x,y).
\end{equation}
Moreover in Section \ref{s:dbcp} we prove that the pivotal measures are integrable on the diagonal and uniformly bounded away from the diagonal: 

\begin{proposition}[Integrability and uniform bound for the critical point intensity]
\label{p:qcb}
There exists a constant $c > 0$, depending only on the field $f$ and the level $\ell \in \R$, such that, for all $(F_1, F_2) \in \mathcal{F}^{(1,2)}$,
\begin{equation}
\label{e:qcb1}
\sup_{x \in F_1} \int_0^1  \int_{\substack{y \in  F_2 \\ |x-y| \le 1 }} \!  I_t(x,y) \, dv_{F_y}(y) dt \le c  
\end{equation} 
and
\begin{equation}
\label{e:qcb2}
 \sup_{t \in [0,1) } 
 \sup_{\substack{(x,y) \in F_1 \times F_2 \\ |x-y| \ge 1  }} \! \! \! I_t(x,y) \le c   .
\end{equation}
\end{proposition}

In the following section we shall need a slight generalisation of this integrability, which we state now. This applies to a `stationary' version of pivotal intensity, i.e.\ one that does not depend on the boxes $B_1$ and $B_2$. Specifically, for $x,y \in \R^d$, $t \in (0,1]$, and $k \in \N$, we define
\begin{equation}
\label{e:hatik}
\hat{I}^{(k)}_t(x,y) = \hat{\varphi}_t(x,y)\widehat{\E} \big[ |\nabla^2 f(x) \nabla^2 f^t(y) |^k \big]^{1/k}
\end{equation}
where $\hat{\varphi}_t(x,y)$ denotes the density of the conditioning
\begin{equation}
\label{e:cond2} (f(x),f^t(y),\nabla f(x), \nabla f^t(y)) = (\ell,\ell,0,0),
\end{equation}
and $\widehat{\E}$ denotes expectation under this conditioning. Abbreviating $\hat{I}_t(x,y) := \hat{I}^{(1)}_t(x,y)$, we note that if $x$ and $y$ lie in the interior of the boxes $B_1$ and $B_2$ respectively, then $\hat{I}_t(x,y) = \hat{I}_t(0,x-y)$ coincides with the pivotal intensity $I_t(x,y)$ defined in \eqref{e:CritInt}.

\begin{proposition}[Bounds for the stationary critical point intensity]
\label{p:qcb2}
There exists a constant $c > 0$, depending only on the field $f$, the level $\ell \in \R$, and $k \in \N$, such that
\[ \int_0^1  \int_{ |u| \le 1 } \!  \hat{I}^{(k)}_t(0,u) \, du dt \le c  
 \qquad \text{and} \qquad  \sup_{t \in (0,1] } \sup_{ |u| \ge 1  }  \hat{I}^{(k)}_t(0,u) \le c   . \]
\end{proposition}
\begin{remark}
We only use this proposition in the case $k=1,2$, but formulate its general statement to convey that $k \leq 2$ plays no role in the proof
\end{remark}

\subsubsection{The covariance formula for component count exceedances}

As mentioned above, we deduce the variance formula in Theorem \ref{t:vf} as a consequence of a covariance formula for `topological events' established in \cite{bmr20}; in particular we shall apply this formula to the component count tail events $\{N \ge k_1\}$ and $\{N^t \ge k_2\}$.

For $k \in \N$ define
\[ P^+_k =   \{ (k-1, k) \} \subset P^+  \quad  \text{and} \quad  P^-_k =    \{ (k, k-1) \} \subset P^- , \]
and $P_k = P^+_k \cup P^-_k$.

\begin{proposition}[Special case of {\cite[Theorem 2.14]{bmr20}}]
\label{p:cf}
For each $t \in [0,1)$ and $k_1, k_2 \in \N$,
\[ \Cov \big[ \id_{N \ge k_1}, \id_{ N^t \ge k_2} \big] =  \int_0^t  \int_{\Lambda_R \times \Lambda_R} \! \! \! K(x-y)  \! \! \sum_{\iota_1 \in P_{k_1} , \iota_2 \in P_{k_2} } \sigma(\iota_1)\sigma(\iota_2)  \, dp^{\iota_1, \iota_2}_s(x,y) \,  ds ,\]
recalling that $\sigma(\iota) = \pm 1$ if  $\iota \in P^\pm$.
\end{proposition}
\begin{remark}
In Proposition \ref{p:cf} we could have omitted the strata $\mathcal{F}_0^{(1)},\mathcal{F}_0^{(2)}$ of dimension $0$ from the definition of the pivotal measures, because critical points on such strata are never pivotal for the component count (i.e\ if the field is conditioned to have a critical point on a dimension $0$ strata, then almost surely  $A(\iota_1,\iota_2)$ occurs for some $\iota_1,\iota_2 \in P^0$). However this is a specific feature of the way we defined the component count (i.e.\ excluding boundary components) and the convexity of $B_1,B_2$, so we prefer not to do this.
\end{remark}
\begin{proof}
This is \cite[Theorem 2.14]{bmr20} applied to the stratified domains $B_1,B_2$ (which is a `tame stratification' in the sense of \cite{bmr20}) and the topological events $\{N \ge k_1\}$ and $\{N \ge k_2\}$, since one can check (recalling Lemma \ref{l:change}) that the `pivotal intensities' $I^+_t(x,y)$ and $I^-_t(x,y)$ as defined in \cite[Theorem 2.14]{bmr20} coincide, respectively, with 
\[  \sum_{\iota_1 \in P_{k_1} , \iota_2 \in P_{k_2}} \id_{\sigma(\iota_1) = \sigma(\iota_2) }   I^{\iota_1, \iota_2}_t(x,y)    \qquad \text{and} \qquad   \sum_{\iota_1 \in P_{k_1} , \iota_2 \in P_{k_2} }  \id_{\sigma(\iota_1) \neq \sigma(\iota_2) }  I^{\iota_1, \iota_2}_t(x,y)   . \]
Strictly speaking an application of \cite[Theorem 2.14]{bmr20} only gives the limiting case $t \uparrow 1$ of this formula, that is,
\begin{align*}
 \Cov \big[ \id_{N \ge k_1}, \id_{ N \ge k_2} \big]  & =  \lim_{t \uparrow 1} \Cov \big[ \id_{N \ge k_1}, \id_{ N^t \ge k_2} \big]  \\
 &  = \lim_{t\uparrow 1}\int_0^t  \int_{\Lambda_R \times \Lambda_R} \! \! \! K(x-y)  \! \! \sum_{\iota_1 \in P_{k_1} , \iota_2 \in P_{k_2} } \sigma(\iota_1)\sigma(\iota_2)  \, dp^{\iota_1, \iota_2}_s(x,y) \,  ds  , 
 \end{align*}
 but the desired formula for all $t \in [0,1)$ is established in the final step of the proof of \cite[Theorem 2.14]{bmr20} immediately before taking the $t \uparrow 1$ limit.
\end{proof}

\subsubsection{Proof of the covariance formula} To deduce Theorem~\ref{t:cf} from Proposition \ref{p:cf} we make use of the integrability of the critical point measure
\begin{equation}
\label{e:qcbweak}
\int_{0}^1\int_{B_1\times B_2}dp_t(x,y)dt=\int_{0}^1\int_{B_1\times B_2}\sum_{F_1,F_2\in\mathcal{F}}I_t(x,y)dv_{F_1}(x)dv_{F_2}(y)dt<\infty
\end{equation}
which follows from \eqref{e:qcb1}--\eqref{e:qcb2} above.

We also make use of the first stability property from the previous section applied to the interpolation $f^t$ and the component count on the box $B_2$:

\begin{lemma}
\label{l:qual_stab}
Let $f$ satisfy Assumption~\ref{a:gen}, then $N^t_\diamond(B_2) \to N_\diamond(B_2)$ almost surely as $t \to 1$.
\end{lemma}
\begin{proof}
By Lemma~\ref{l:nondegen1}, $(f(x),\nabla f(x))$ is non-degenerate for any $x$. Combining this with the fact that $f$ is $C^2$-smooth, we may apply Bulinskaya's lemma (\cite[Lemma~11.2.10]{at07}), which states that $f$ has no (stratified) critical points on $B_2$ at level $\ell$ almost surely. (To be specific, this follows from applying Bulinskaya's lemma on each stratum of $B_2$.) The result then follows from Lemma \ref{l:Topological_stability}.
\end{proof}

\begin{proof}[Proof of Theorem \ref{t:cf}]
Recalling that $N^1 = N_\diamond(B_2)$, we seek a formula for $\Cov[N,N^1]$.

We first establish that, for every $t \in [0, 1)$,
\begin{equation}
    \label{e:vfproof1}
\Cov [ N, N^t]=\int_0^t  \int_{B_1\times B_2}  \! \! \! K(x-y) \Big( d p_s^{++}(x,y) + d p_s^{--}(x,y) - d p_s^{+-}(x,y) - d p_s^{-+}(x,y)  \Big) ds .
\end{equation}
For this we fix a truncation parameter $\tau\in\N$ and consider the variables
\[ \min\{N, \tau\} = \sum_{k=1}^\tau \id_{N \ge k} \quad \text{and} \quad \min\{N^t, \tau\} = \sum_{k=1}^\tau \id_{N^t \ge k}  . \]
Applying Proposition \ref{p:cf}, and interchanging the finite summation with the integrals, we deduce that
\begin{equation}
\label{e:vfproof2}
 \Cov \big[ \min\{N, \tau\}, \min\{ N^t, \tau \} \big]= \int_0^t  \int_{B_1 \times B_2} \! \! \! K(x-y)  \! \! \sum^\tau_{k_1,k_2 = 1} \sum_{\iota_1 \in P_{k_1} , \iota_2 \in P_{k_2} } \sigma(\iota_1)\sigma(\iota_2)  \, dp^{\iota_1, \iota_2}_s(x,y) \,  ds .
\end{equation}
Now consider a pair of strata $(F_1, F_2) \in \mathcal{F}^{(1,2)}$. For every $(x,y) \in F_1 \times F_2$ and $s\in[0,1)$, we have
\[ \lim_{\tau \to \infty}  \sum^\tau_{k_1,k_2 = 1} \sum_{\iota_1 \in P_{k_1} , \iota_2 \in P_{k_2} } \sigma(\iota_1)\sigma(\iota_2)  \, I^{\iota_1, \iota_2}_s(x,y)  = \sum_{\iota_1,\iota_2 \in P} \sigma(\iota_1)\sigma(\iota_2)  \, I^{\iota_1, \iota_2}_s(x,y) , \]
and also, by \eqref{e:intbound}, the uniform bound 
\[ \Big| \sum^\tau_{k_1,k_2 = 1} \sum_{\iota_1 \in P_{k_1} , \iota_2 \in P_{k_2} } \sigma(\iota_1)\sigma(\iota_2)  \, I^{\iota_1, \iota_2}_s(x,y) \Big| \le I_s(x,y) , \]
which is integrable on $(s,x,y) \in [0,t] \times F_1 \times F_2$ by \eqref{e:qcbweak}. Hence by dominated convergence we deduce that the limit of \eqref{e:vfproof2} as $\tau\to\infty$ is equal to
\begin{equation}\label{e:vfproof6}
\begin{aligned}
 &  \int_0^t  \int_{B_1\times B_2} \! \! \! K(x-y)  \! \! \sum_{\iota_1,\iota_2 \in P} \sigma(\iota_1)\sigma(\iota_2)  \, dp^{\iota_1, \iota_2}_s(x,y) \,  ds  \\
&  \qquad\qquad\qquad = \int_0^t  \int_{B_1\times B_2}  \! \! \! K(x-y) \Big( d p_s^{++}(x,y) + d p_s^{--}(x,y) - d p_s^{+-}(x,y) - d p_s^{-+}(x,y)  \Big) ds  ,
\end{aligned}
\end{equation}
where in the last equality we used \eqref{e:sumint2}. Since $\E[N^2],\E[(N^t)^2] < \infty$, taking $\tau \to \infty$ on both sides of \eqref{e:vfproof2} establishes~\eqref{e:vfproof1}.

In light of \eqref{e:vfproof1}, it remains to prove that 
\begin{equation}
    \label{e:vfproof3}
 \lim_{t \uparrow 1} \Cov \big[ N, N^t\big] = \Cov[N, N^1] .
\end{equation} 
By our argument above,
\begin{equation}\label{e:vfproof5}
    \lim_{t \uparrow 1} \Cov \big[ N, N^t\big]=\lim_{t \uparrow 1}\lim_{\tau\to\infty} \Cov \big[ \min\{N,\tau\}, \min\{N^t,\tau\}\big].
\end{equation}
Consider taking these limits in the opposite order: by Lemma \ref{l:qual_stab} $N^t \to N^1$ almost surely, then applying the dominated convergence theorem twice (using square integrability of $N,N^1$) we see that
\begin{displaymath}
\lim_{\tau\to\infty}\lim_{t\uparrow 1}\Cov \big[ \min\{N,\tau\}, \min\{N^t,\tau\}\big]=\lim_{\tau\to\infty}\Cov\big[\min\{N,\tau\},\min\{N^1,\tau\}\big]=\Cov\big[N, N^1\big].
\end{displaymath}
Therefore by \eqref{e:vfproof5}, we can establish \eqref{e:vfproof3} by justifying this interchange of limits. Let us denote
\begin{displaymath}
g(t,\tau)=\Cov \big[ \min\{N,\tau\}, \min\{N^t,\tau\}\big],    
\end{displaymath}
then by the Moore-Osgood theorem we can say that $\lim_{\tau\to\infty}\lim_{t\uparrow 1}g(t,\tau)=\lim_{t\uparrow 1}\lim_{\tau\to\infty}g(t,\tau)$ provided that
\begin{enumerate}
    \item $\lim_{\tau\to\infty}g(t,\tau)$ exists for every $t\in [0,1)$, and
    \item $\lim_{t\uparrow1}g(t,\tau)$ exists uniformly over $\tau$.
\end{enumerate}
We have already shown that $\lim_{\tau\to\infty}g(t,\tau)$ is given by \eqref{e:vfproof6} for each $t$, verifying the first requirement. Turning to the second requirement, let $\tau\in\N$ and $0 \le t_1<t_2<1$, by \eqref{e:vfproof2} and \eqref{e:intbound}
\begin{align*}
    \lvert g(t_1,\tau)-&g(t_2,\tau)\rvert\leq\int_{t_1}^{t_2}\int_{B_1\times B_2}\lvert K(0)\rvert\sum^\tau_{k_1,k_2 = 1} \sum_{\iota_1 \in P_{k_1} , \iota_2 \in P_{k_2} } dp^{\iota_1, \iota_2}_s(x,y) \,  ds\\
    &\leq \lvert K(0)\rvert\int_{t_1}^{t_2}\int_{B_1\times B_2}\sum^\infty_{k_1,k_2 = 1} \sum_{\iota_1 \in P_{k_1} , \iota_2 \in P_{k_2} } dp^{\iota_1, \iota_2}_s(x,y) \,  ds\leq \lvert K(0)\rvert\int_{t_1}^{t_2}\int_{B_1\times B_2}dp_s(x,y) ds.
\end{align*}
Hence, by \eqref{e:qcbweak}, $g(t,\tau)$ is uniformly (in $\tau$) Cauchy as $t\uparrow 1$ justifying the Moore-Osgood theorem. Combining \eqref{e:vfproof1} and \eqref{e:vfproof3} completes the proof of Theorem~\ref{t:cf}.
\end{proof}

\subsection{Specialisation to the variance and the general upper bound}
We next specialise the formula to the case of the variance and deduce the general upper bound, 
 stated in Theorems \ref{t:vf} and~\ref{t:vub} respectively.

Compared to the covariance formula, the specialisation to the variance has additional monotonicity and convexity properties. These are essentially due to Chatterjee, who proved this result for functions of finite-dimensional Gaussian vectors \cite{cha08, cha14}. Here we state a general version for continuous (not necessarily stationary) Gaussian fields:

\begin{proposition}[Monotonicity and convexity of the interpolation]
\label{p:mon}
Let $D \subseteq \R^d$, let $f$ and $\tilde{f}$ be independent copies of a continuous Gaussian field on $D$, and define the interpolation  $f^t = t f  + \sqrt{1-t^2} \tilde{f}$, $t \in [0,1]$. Let $G : C^0(D) \to \R$ be such that $\E[G(f)^2] < \infty$. Then
\[ t \mapsto \Cov[G(f), G(f^t)] \ , \quad t \in [0,1] , \]
is non-decreasing and convex.
\end{proposition}

\begin{proof}
As mentioned above, the corresponding result in the finite-dimensional setting is due to Chatterjee. More precisely, let $X$ and $\widetilde{X}$ be $n$-dimensional i.i.d.\ standard Gaussian vectors, define $X^t = t X + \sqrt{1-t^2} \widetilde{X}$, $t \in [0,1]$, and let $G : \R^n \to \R$ be absolutely continuous with 
 $G(X) \in L^2$ and $\|\nabla G(X)\|_2 \in L^2$. Then \cite[Lemma 3.5]{cha08} states that 
 \[ t \mapsto \Cov[G(X),G(X^t)] \quad \text{and} \quad  t \mapsto \Cov[ \partial G(X) / \partial X_i , \partial G(X^t) / \partial X^t_i] \  ,  \quad 1 \le i \le n, \]
 are non-decreasing on $[0,1]$ (note that \cite[Lemma 3.5]{cha08} uses a different parameterisation for the interpolation). Recalling that (see Proposition \ref{p:gcf})
\begin{equation}
\label{e:gcf2}
\Cov[G(X),G(X^t)] = \int_0^t \sum_{1 \le i \le n} \E \Big[  \frac{\partial G(X)}{\partial X_i} \frac{\partial G(X^s)}{\partial X^s_i}   \Big] \, ds 
\end{equation}
this proves the result.

The extension to a general Gaussian vector $X$ follows immediately by writing $X = Q Z$, where $Q$ is a square-root of the covariance matrix of $X$ and $Z$ is an i.i.d.\ Gaussian vector, and applying the result to $G \mapsto G \circ Q$.

It remains to extend the result to arbitrary functions of continuous Gaussian fields. For this we will repeatedly use the fact that the pointwise limit of a sequence of non-decreasing and convex functions on $[0,1]$ is also non-decreasing and convex. 

First we reduce to the case of bounded $G$ by taking the limit as $\tau \to \infty$ of the maps
\[ t \mapsto  \Cov[ G(f) \id_{|G(f)| \le \tau}, G(f^t) \id_{|G^t(f)| \le \tau}] \ , \quad t \in [0,1], \]
which converge pointwise to $t \mapsto \Cov[G(f), G(f^t)]$ since $G(f), G(f^t) \in L^2$.

Next we reduce to the finite-dimensional case. Let $D_n$ be a finite subset of $D$ such that $D_n\subseteq D_{n+1}$ and $\cup_{n\in\N}D_n$ is dense in $D$. We then define the random variables
\[ G_n = \E \big[ G(f)   \, | \, f|_{D_n} \big] \quad \text{and} \quad  G^t_n = \E \big[ G(f^t)  \, |\, f^t|_{D_n} \big]  .\]
Since $(G_n)_{n \in \N}$ (resp.\ $(G^t_n)_{n \in \N}$) is a bounded martingale, by martingale convergence there exists a random variable $G_\infty$ (resp.\ $G^t_\infty$) such that $G_n \to G_\infty$ (resp.\ $G^t_n \to G^t_\infty$) almost surely and in~$L^2$. By the continuity of $f$, $G_\infty = G(f)$ and $G^t_\infty = G(f^t)$ almost surely, and we conclude that as $n \to \infty$ the maps
\[ t \mapsto  \Cov[ G_n, G^t_n ] \ , \quad t \in [0,1], \]
converge pointwise to $t \mapsto \Cov[G(f), G(f^t)]$. 

Finally we reduce to the case of smooth $G : \R^n \to R$ with bounded derivatives of all orders by defining, for $\eps > 0$, the mollification 
\[ G_\eps(x) = \E[ G(x + \eps \bar{X} ) ] = (G \ast \varphi_\eps)(x) , \]
where $\bar{X}$ is an independent i.i.d.\ standard Gaussian vector, $\varphi_\eps$ is the density of $\eps \bar{X}$, and $\ast$ denotes convolution. Since $\varphi_\eps$ is smooth with derivatives in $L^1(\R^n)$ and $G$ is bounded, by Young's convolution inequality $G_\eps$ is smooth with bounded derivatives of all order. Moreover, since $G$ is bounded, $G_\eps(x) \to G(x)$ for almost every $x$, and also $G_\eps(X) \to G(X)$ and $G_\eps(X^t) \to G(X^t)$ in $L^2$. Hence we deduce that, as $\eps \to 0$, the maps
\[  t \mapsto  \Cov[G_\eps(X) , G_\eps(X^t) ] \ , \quad t \in [0,1]  , \]
converge pointwise to $t \mapsto \Cov[G(X),G(X^t)]$, which completes the proof.
\end{proof}

\begin{remark}
 One can deduce from the formula \eqref{e:gcf2} that, in the finite-dimensional case, $t \mapsto \Cov[G(X), G(X^t)]$ is absolutely continuous if $G$ is absolutely continuous with $\|\nabla G(X)\|_2 \in L^2$, although unlike the monotonicity and convexity it is not clear that this property transfers to the continuum limit.
\end{remark}

\begin{remark}
The monotonicity properties in Proposition \ref{p:mon} are special to the case of the variance, i.e.\ it is not true in general that
\[ t \mapsto \Cov[G_1(X), G_2(X^t)] \ , \quad t \in [0,1] , \]
is non-decreasing and convex, even for finite-dimensional $X$.
\end{remark}

\begin{proof}[Proof of Theorem~\ref{t:vf}]
From the proof of Theorem~\ref{t:cf} (specifically \eqref{e:vfproof1}) in the case that $\star=\diamond$ and $B_1=B_2=\Lambda_R$, we see that $ t\mapsto\Cov[N,N^t]$ is absolutely continuous with derivative~\eqref{e:der}. By Proposition~\ref{p:mon} this map is also convex and non-decreasing, completing the proof of the theorem.
\end{proof}

\begin{proof}[Proof of Theorem \ref{t:vub}]
Recall that Corollary \ref{c:vf} states that
\[ \Var[N_\star(R)]   \le    \int_{0}^1  \int_{\Lambda_R \times \Lambda_R}  \! \! |K(x-y)|  d p_t(x,y)   \, dt   .  \] 
Summing over pairs of strata $F_1,F_2 \in \mathcal{F}$, partitioning the integral over $(x,y) \in F_1 \times F_2$ according to whether or not $|x-y| \le 1$, and applying Proposition \ref{p:qcb}, the above is at most
\begin{multline*}
\sum_{F_1, F_2 \in \mathcal{F}}  K(0)\int_{0}^1  \int_{\substack{(x,y) \in F_1 \times F_2 \\ |x-y| \le 1 }}  I_t(x,y) \,  dv_{F_1}(x) dv_{F_2}(y) \, dt  \\
+    \int_{0}^1  \int_{\substack{(x,y) \in F_1 \times F_2 \\ |x-y| \ge 1 }}   |K(x-y)| I_t(x,y)  \, dv_{F_1}(x) dv_{F_2}(y)  \,  dt \\
\leq \sum_{F_1, F_2 \in \mathcal{F}}  K(0)\int_{x\in F_1}c\,  dv_{F_1}(x) +    \int_{(x,y) \in F_1 \times F_2}   c\lvert K(x-y)\rvert\, dv_{F_1}(x) dv_{F_2}(y) \\
\le c_1 \Big(\max\{1,R^d\}+  \sum_{F_1, F_2 \in \mathcal{F}} \int_{F_1 \times F_2} |K(x-y)| dv_{F_1}(x) dv_{F_2}(y)  \Big)  ,
\end{multline*}
for a constant $c_1 > 0$. Moreover there is a $c_d > 0$ depending only on the dimension such that, for all $R \ge 1$,
\begin{displaymath}
    \sum_{F_1, F_2 \in \mathcal{F}} \int_{F_1 \times F_2} |K(x-y)| dv_{F_1}(x) dv_{F_2}(y) \le c_d \int_{\Lambda_R \times \Lambda_R} |\widetilde{K}(x-y)| dx dy  \le c_d R^d \int_{\Lambda_{2R}} |\widetilde{K}(x)| dx .
\end{displaymath}
Combining the above gives the result.
\end{proof}

\subsection{Asymptotic formula}\label{ss:Asymptotic}
In this subsection we prove the asymptotic formula in Theorem~\ref{t:Asymptotic}. We begin by defining stationary versions $\hat{I}_t^{\pm \pm}$ of the pivotal intensities $I_t^{\pm \pm}$ introduced in \eqref{e:pivint} above; the key difference will be that, unlike the latter, the former will be independent of the box $\Lambda_R$.

Recall from Lemma~\ref{l:change} that each critical point in the interior of a box $B$ can be classified as either positively, negatively, or not pivotal, and by Lemma~\ref{l:AsympStab} this classification is the same for all sufficiently large boxes. We shall need this statement for the Gaussian fields $f$ and $f^t$ under the conditioning \eqref{e:cond}:

\begin{lemma}
\label{l:asymptotic}
Assume that $K\in L^1(\R^d)$. Let $x,y \in \R$ and $t \in [0,1)$. Then with probability one, for any realisation of
\begin{equation}\label{e:CondField}
(f,f^t\;|\; f(x)=f^t(y)=\ell,\nabla f(x)=\nabla f^t(y)=0)
\end{equation}
the following holds: there exists $r_x,r_y>0$ and $\Delta^t_x,\Delta^t_y\in\{+,-,0\}$ such that for all boxes $B$ containing $x+\Lambda_{r_x}$ and $y+\Lambda_{r_y}$, $x$ is $(\Delta^t_x,B)$-pivotal for $f$ and $y$ is $(\Delta^t_y,B)$-pivotal for $f^t$.
\end{lemma}
\begin{proof}
By Lemma~\ref{l:AsympStab} it is enough to show that the critical points of $f$ and $f^t$ at $x$ and $y$ respectively are almost surely non-degenerate. Since $K \in L^1(\R^d)$, the spectral measure $\mu$ has a density, and in particular its support contains an open set. Hence, by Lemmas \ref{l:nondegen1} and \ref{l:nondegen2}, the Gaussian vector
\[ (\nabla^2f(x) \;|\; f(x)=f^t(y)=\ell,\nabla f(x)=\nabla f^t(y)=0) \]
is non-degenerate. (Note that here we view $\nabla^2 f(x)$ as the $d(d+1)/2$-dimensional Gaussian vector $(\partial_{ij}f(x)\;|\;i\leq j)$ so as to avoid the trivial degeneracy $\partial_{ij}f(x)=\partial_{ji}f(x)$. This convention is standard in the literature on smooth Gaussian fields.) The determinant of a symmetric $d\times d$ matrix, viewed as a function of $d(d+1)/2$ variables, is a polynomial and so its zero-set has Lebesgue measure zero in $\R^{d(d+1)/2}$. Therefore with probability one
    \begin{displaymath}
    \P\left( | \nabla^2f(x)| =0\;|\;f(x)=f^t(y)=\ell,\nabla f(x)=\nabla f^t(y)=0\right)=0.
    \end{displaymath}
    The same argument applies if we replace $\nabla^2f(x)$ by $\nabla^2 f^t(y)$, completing the proof of the lemma.
\end{proof}

We define the \textit{stationary pivotal intensities} as
\[ \hat{I}_t^{\pm\pm}(x,y)= \hat{\varphi}_t(x,y) \widehat{\E} \big[ |\nabla^2 f(x) \nabla^2 f^t(y) | \ind_{(\Delta_x^t,\Delta_y^t)=(\pm\pm)}\big], \]
recalling that $\hat{\varphi}_t(x,y)$ denotes the density of the conditioning \eqref{e:cond2} and $\widehat{\E}$ denotes expectation under this conditioning.

As already mentioned, in the proof of Theorem~\ref{t:Asymptotic} we shall compare $\hat{I}_t^{\pm\pm}(x,y)$ and $I_t^{\pm \pm}(x,y)$ for $x,y$ lying in the bulk of $\Lambda_R$. To this end, given $s > 0$, $x,y \in \R^d$, a realisation of \eqref{e:CondField}, and $r_x^*,r_y^*$ defined as the infima of $r_x,r_y$ in Lemma~\ref{l:asymptotic}, denote
\begin{displaymath}
q_t(x,y,s)=\P(\max\{r_x,r_y\}>s).
\end{displaymath}
So roughly speaking this is the probability that boxes $\Lambda_s$, centred at $x$ and $y$ respectively, are not big enough to capture the change in topology at the pivotal points $x,y$. Clearly this is non-increasing in $s$ and, by Lemma~\ref{l:asymptotic}, tends to zero as $s\to\infty$ for each choice of $t,x,y$. Also, by stationarity, $q_t(x,y,s)=q_t(0,x-y,s)$. 

Let us now observe an equivalent definition of the pivotal intensities $I_t^{\pm \pm}(x,y)$ for $x,y$ in the interior of $\Lambda_R$. Given a realisation of \eqref{e:CondField} we write $\Delta_x^t(\Lambda_R)\in\{+,-,0\}$ for the symbol such that $x$ is a $(\Delta_x^t(\Lambda_R),\Lambda_R)$-critical point, and we define $\Delta_y^t(\Lambda_R)$ analogously. Then for $x,y$ in the interior of $\Lambda_R$ we can equivalently define $I_t^{\pm\pm}(x,y)$ as
\[   I_t^{\pm\pm}(x,y) = \hat{\varphi}_t(x,y) \widehat{\E} \big[ |\nabla^2 f(x) \nabla^2 f^t(y) | \ind_{(\Delta_x^t(\Lambda_R),\Delta_y^t(\Lambda_R))=(\pm\pm)}\big] . \]

Now suppose $ r \in (0,R)$ and let $x,y \in\Lambda_{R-r}$. Then, applying the Cauchy-Schwarz inequality to the definitions above,
\begin{equation}\label{e:Station_Comp}
\lvert I_t^{\pm\pm}(x,y)-\hat{I}_t^{\pm\pm}(x,y)\rvert\leq \hat{I}^{(2)}_t(x,y) q_t(x,y,r)^{1/2} = \hat{I}^{(2)}_t(0,x-y) q_t(0,x-y,r)^{1/2},
\end{equation}
where $\hat{I}^{(k)}_t$ is defined in \eqref{e:hatik}. 

\begin{proof}[Proof of Theorem~\ref{t:Asymptotic}]
In the proof $c > 0$ will denote constants that are independent of $R$ and may change from line to line.

First we show that the contribution from boundary strata to the variance formula in \eqref{e:vf} is of order $R^{d-1}$. Let $F_1,F_2$ be two strata of $\Lambda_R$ one of which has dimension at most $d-1$. Then by the fact that $I_t^{\pm\pm}\leq I_t$ and Proposition~\ref{p:qcb}
\begin{multline*}
\left\lvert\int_0^1\int_{F_1\times F_2}K(x-y)\sum_{i,j\in\{+,-\}}\sigma(i)\sigma(j) I_t^{ij}(x,y) \, dv_{F_1}(x)dv_{F_2}(y)dt\right\rvert\\
    \leq 4 K(0)\int_0^1\int_{F_1\times F_2}I_t(x,y) \, dv_{F_1}(x)dv_{F_2}(y)dt\leq cR^{d-1}.
\end{multline*}
Summing over all such pairs of strata, by \eqref{e:vf} we see that
\begin{equation}\label{e:Asymp1}
\Var[N_\star(R)]=\int_0^1\int_{\text{int}(\Lambda_R)^2}K(x-y)\sum_{i,j\in\{\pm\}}\sigma(i)\sigma(j) I_t^{ij}(x,y)\, dxdydt +O(R^{d-1})
\end{equation}
as $R\to\infty$, as claimed.

Next we choose $r=\sqrt{R}$ (in fact any $1\ll r\ll R$ would do) and argue that, in \eqref{e:Asymp1}, one can replace the domain of integration $\text{int}(\Lambda_R)^2$ by $\Lambda_{R-r}^2$ with an $o(R^d)$ error. Indeed using $I_t^{\pm\pm}\leq I_t$ and Proposition~\ref{p:qcb} once more,
\begin{align}\label{e:Asymp2}
& \left\lvert\int_0^1\int_{(\text{int}(\Lambda_R)\setminus\Lambda_{R-r}) \times \text{int}(\Lambda_R)}K(x-y)\sum_{i,j\in\{\pm\}}\sigma(i)\sigma(j) I_t^{ij}(x,y)\, dx dy dt\right\rvert\\
&  \nonumber \qquad \leq 4 K(0) \int_0^1\int_{\substack{ (\text{int}(\Lambda_R)\setminus\Lambda_{R-r}) \times \text{int}(\Lambda_R) \\ \lvert x-y\rvert\leq 1}} I_t(x,y)\, dx dy dt\\
&\nonumber\qquad\qquad\qquad\qquad\qquad\qquad+ c\int_0^1\int_{\substack{(\text{int}(\Lambda_R)\setminus\Lambda_{R-r}) \times \text{int}(\Lambda_R)\\ \lvert x-y\rvert \ge 1}}\lvert K(x-y)\rvert \, dx dy dt.
\end{align}
Note that, in the above expression, the critical point intensity $I_t(x,y)$ is equal to its stationary version $\hat{I}_t(x,y)$,  since $x,y$ are in the interior of $\Lambda_R$. Therefore by the change of variables $u=x-y$ and stationarity, \eqref{e:Asymp2} is bounded above by
\begin{displaymath}
4 K(0)\int_{\text{int}(\Lambda_R)\setminus\Lambda_{R-r}}\int_0^1\int_{\lvert u\rvert\leq 1} \hat{I}_t(0,u) \, du dt dx+c\int_{\text{int}(\Lambda_R)\setminus\Lambda_{R-r}}\int_{\R^d} \lvert K(u)\rvert \, du dx \leq c rR^{d-1}=o(R^d),
\end{displaymath}
where we used Proposition~\ref{p:qcb} and the fact that $K \in L^1(\R^d)$ for the final inequality. Combining with \eqref{e:Asymp1} we have shown that, as $R \to \infty$,
\begin{equation}
\label{e:Asymp1b}
\Var[N_\star(R)]=\int_0^1\int_{\Lambda_{R-r}^2}K(x-y)\sum_{i,j\in\{\pm\}}\sigma(i)\sigma(j) I_t^{ij}(x,y)\, dxdydt +o(R^d)
\end{equation}

A similar argument shows that, as $R \to \infty$,
\begin{align}
\label{e:Asymp1c}
& \int_0^1 \int_{\Lambda_{R-r} 
 \times \R^d}K(x-y)\sum_{i,j\in\{\pm\}}\sigma(i)\sigma(j) \hat{I}_t^{ij}(x,y)\, dx dt  \\
\nonumber & \qquad\qquad\qquad\qquad =  \int_0^1\int_{\Lambda_{R-r}^2}K(x-y)\sum_{i,j\in\{\pm\}}\sigma(i)\sigma(j) \hat{I}_t^{ij}(x,y)\, dxdydt +o(R^d).
\end{align}
Indeed using $\hat{I}_t^{\pm\pm}\leq \hat{I}_t$ and Proposition~\ref{p:qcb} again,
\begin{align}\label{e:Asymp3}
& \left\lvert\int_0^1\int_{\Lambda_{R-r} \times (\R^d \setminus \Lambda_{R-r})}K(x-y)\sum_{i,j\in\{\pm\}}\sigma(i)\sigma(j) \hat{I}_t^{ij}(x,y)\, dx dy dt\right\rvert\\
&  \nonumber \qquad  \leq 4 K(0) \int_0^1\int_{\substack{ \Lambda_{R-r} \times (\R^d \setminus \Lambda_{R-r}) \\ \lvert x-y\rvert\leq 1}} \hat{I}_t(x,y)\, dx dy dt + c\int_0^1\int_{\substack{\Lambda_{R-r} \times (\R^d \setminus \Lambda_{R-r})\\ \lvert x-y\rvert \ge 1}}\lvert K(x-y)\rvert \, dx dy dt.
\end{align}
By the change of variables $u=x-y$ and stationarity, \eqref{e:Asymp3} is bounded above by
\begin{align*}
& 4 K(0)\int_{ \Lambda_{R-r} \setminus \Lambda_{R-2r}}\int_0^1\int_{\lvert u\rvert\leq 1} \hat{I}_t(0,u) \, du dt dx \\
& \qquad\qquad\qquad\qquad +c \int_{(\Lambda_{R-r} \setminus \Lambda_{R-2r} ) \times \R^d }  \lvert K(x-y)\rvert \, dx dy   +c \int_{ \Lambda_{R-2r} \times ( \R^d \setminus \Lambda_{R-r}  ) }  \lvert K(x-y)\rvert \, dx dy   \\
& \qquad \qquad\qquad\qquad\qquad\qquad\qquad \le crR^{d-1} + c R^d \int_{|u| \ge r }  \lvert K(u)\rvert \, du  =o(R^d).
\end{align*}

It remains to replace, in \eqref{e:Asymp1b}, the pivotal intensities $I_t^{\pm \pm}$ with their stationary versions $\hat{I}_t^{\pm\pm}$ up to $o(R^d)$ error, so as to compare with \eqref{e:Asymp1c}. Using \eqref{e:Station_Comp},
\begin{align}
\label{e:Asymp4}
& \left\lvert\int_0^1\int_{\Lambda_{R-r}^2}K(x-y)\sum_{i,j\in\{\pm\}}\sigma(i)\sigma(j) (I_t^{ij}(x,y)-\hat{I}_t^{ij}(x,y))\, dx dy dt\right\rvert\\
\nonumber & \qquad\qquad\qquad
    \leq\int_0^1\int_{\Lambda_{R-r}^2}\lvert K(x-y)\rvert\sum_{i,j\in\{\pm\}}\lvert I_t^{ij}(x,y)-\hat{I}_t^{ij}(x,y)\rvert\, dx dy dt\\
  \nonumber  & \qquad \qquad \qquad\qquad\qquad   \leq\int_0^1\int_{\Lambda_{R-r}^2}\lvert K(x-y)\rvert \hat{I}^{(2)}_t(0,x-y)q_t(0,x-y,r)^{1/2}\, dx dy dt,
\end{align}
which is bounded above by
\[ R^{d}\int_0^1\int_{\R^d}\lvert K(u)\rvert \hat{I}_t^{(2)}(0,u)q_t(0,u,r)^{1/2}\, du dt .\]
We next claim the dominated convergence theorem implies that
\begin{displaymath}
    \int_0^1\int_{\R^d}\lvert K(u)\rvert \hat{I}_t^{(2)}(u,0)q_t(u,0,r)^{1/2}\, du dt\to 0
\end{displaymath}
as $R\to\infty$, from which we conclude that \eqref{e:Asymp4} is $o(R^d)$. Indeed since $q_t(u,0,r)\leq 1$ and $q_t(u,0,r)\to 0$ as $r\to\infty$, dominated convergence applies since
\begin{align*}
    \int_0^1\int_{\R^d}\lvert K(u)\rvert\hat{I}_t^{(2)}(0,u) \, dudt    \leq K(0)\int_0^1\int_{\lvert u\rvert\leq 1} \hat{I}_t^{(2)}(0,u) \, dudt+c\int_0^1\int_{\lvert u\rvert \ge 1}\lvert K(u)\rvert\, dudt<\infty,
\end{align*}
where we used Proposition~\ref{p:qcb2} and the integrability of $K$.

Combining the fact that \eqref{e:Asymp4} is $o(R^d)$ with \eqref{e:Asymp1b}--\eqref{e:Asymp1c}, and by stationarity, we obtain
\begin{align*}
\Var[N_\star(R)] & = \int_0^1 \int_{\Lambda_{R-r} 
 \times \R^d}K(x-y)\sum_{i,j\in\{\pm\}}\sigma(i)\sigma(j) \hat{I}_t^{ij}(x,y)\, dx dt +o(R^{d})  \\
& = R^d(1 - o(1) ) \int_0^1\int_{\R^d}K(x)\sum_{i,j\in\{\pm\}}\sigma(i)\sigma(j) \hat{I}_t^{ij}(0,x)\, dx dt + o(R^d) ,
\end{align*}
which completes the proof of \eqref{e:asymptotic} and \eqref{e:sigma}. 
\end{proof}

\section{Analysis of the critical point intensity}
\label{s:dbcp}

In this section we prove the integrability and boundedness of the critical point intensity $I_t$, stated in Proposition~\ref{p:qcb} above, as well as its generalisation in Proposition \ref{p:qcb2}.

\subsection{Reduction to pointwise bounds}
We first reduce the proof of Propositions~\ref{p:qcb} and~\ref{p:qcb2} to pointwise bounds on the two components whose product gives the intensity $I_t$. 

Let $\{e_i\}_{1 \le i \le d}$ denote the standard basis of $\R^d$, and for any non-empty ${I \subseteq \{1,2,\ldots,d\}}$, let $\nabla_I$ denote the gradient with respect to the basis directions  $\{e_i\}_{i \in I}$. For $I_1, I_2 \subseteq \{1,2,\ldots d\}$ and $(t,x) \in [0,1) \times \R^d$ we shall consider the Gaussian vector 
\begin{equation}
\label{e:cond3}
\big( f(0),f^t(x),\nabla_{I_1} f(0), \nabla_{I_2} f^t(x)  \big)  
\end{equation}
and the conditional expectation, for $\ell \in \R$ and $k \in \N$,
\begin{equation}
\label{e:hess3}
\E\big[ | \nabla^2_{I_1} f(0)   \nabla^2_{I_2} f^t(x)   |^k  \, \big| \, f(0) = f^t(x) = \ell, \nabla_{I_1}f(0) = \nabla_{I_2}f^t(x) = 0 \big]  .
\end{equation}
If $I_1$ is empty then by convention we remove the corresponding term $\nabla_{I_1}$ from \eqref{e:cond3} and the conditioning in \eqref{e:hess3}, and set $|\nabla^2_{I_1} f(0)| = 1$ in \eqref{e:hess3}, and similarly if $I_2$ is empty.

For a square-integrable random vector $X$, let $\DC(X) = | \textrm{Cov}[X] |$ denote the determinant of its covariance matrix.

The pointwise bounds are the following:
\begin{lemma}
\label{l:numbound}
There exists a $c > 0$ depending only on the field $f$ such that, for every $I_1,I_2 \subseteq \{1,2,\ldots,d\}$ and $(t,x) \in [0,1) \times \R^d$,
\[ \DC \big( f(0),f^t(x),\nabla_{I_1} f(0), \nabla_{I_2} f^t(x)  \big)  \ge \begin{cases} c & \text{if $|x| \ge 1$,} \\ c  \max\{(1-t)^{1/2}, |x| \}^{2d'} (1-t)  & \text{if $|x| \le 1$,} \end{cases} \]
where $d' = \textrm{dim}(I_1 \cap I_2)$. 
\end{lemma}

\begin{lemma}
\label{l:denombound}
Let $k \in \N$. Then there exists $c > 0$  depending only on the field $f$ and $k$, such that, for every $\ell \in \R$, $I_1,I_2 \subseteq \{1,2,\ldots,d\}$, and $(t,x) \in [0,1) \times \R^d\setminus\{0\}$,
\[  \E\big[ | \nabla^2_{I_1} f(0)   \nabla^2_{I_2} f^t(x)   |^k  \, \big| \, f(0) = f^t(x) = \ell, \nabla_{I_1}f(0) = \nabla_{I_2}f^t(x) = 0 \big]    \le c (1+\ell^{2dk}) . \]
\end{lemma}

\begin{remark}
\label{r:numbound}
Lemma \ref{l:numbound} is a quantification of our previous claim that
\[ \big( f(x),f^t(y),\nabla_{F_x} f(x), \nabla_{F_y} f^t(y)  \big) \]
is non-degenerate for all $t \in [0,1)$, $x \in B_1$, and $y \in B_2$, and in particular the intensity $I_t(x,y)$ is well-defined. In Lemma~\ref{l:nondegen2} below we give a simpler non-quantitative proof of this fact.
\end{remark}

\begin{remark}
Although we only require Lemma \ref{l:denombound} for fixed $\ell \in \R$, we include the dependency since we believe it may prove useful in future work.
\end{remark}

Let us complete the proof of Proposition \ref{p:qcb} and \ref{p:qcb2} given these pointwise bounds:

\begin{proof}[Proof of Proposition \ref{p:qcb}]
Recall that $\ell \in \R$ is fixed. In the proof $c > 0$ will denote constants which depend only on $f$ and $\ell$, and may change from line to line. 

Let $(F_1, F_2) \in \mathcal{F}^{(1,2)}$ be given, and let $I_1$ and $I_2$ denote the indices of the basis directions that span $F_1$ and $F_2$ respectively, with $d_i = \text{dim}(I_i)$. Observe that, by stationarity, for every $(t,x,y) \in [0,1) \times F_1 \times F_2$,
\[ \big( f(x),f^t(y),\nabla_{F_x} f(x), \nabla_{F_y} f^t(y)  \big) \stackrel{d}{=} \big( f(0),f^t(u),\nabla_{I_1} f(0), \nabla_{I_2} f^t(u)  \big)  ,\]
where $u = y-x$.

Noting the trivial bound 
\[ \varphi_t(0,u) \le (2 \pi)^{-(d_1+d_2+2)/2}  \DC \big( f(0),f^t(u),\nabla_{I_1} f(0), \nabla_{I_2} f^t(u)  \big)^{-1/2} \]
on the Gaussian density, we see that
\begin{align}
 \label{e:itbound}
 I_t(x,y) & \le c \times \DC \big( f(0),f^t(u),\nabla_{I_1} f(0), \nabla_{I_2} f^t(u)  \big)^{-1/2}     \\
& \nonumber \qquad \qquad \times  \E\big[ | \nabla^2_{I_1} f(0)   \nabla^2_{I_2} f^t(u)   |  \, \big| \, f(0) = f^t(u) = \ell, \nabla_{I_1}f(0) = \nabla_{I_2}f^t(u) = 0 \big] .  
\end{align}
 Combining this with Lemmas \ref{l:numbound} and \ref{l:denombound}, we immediately obtain the uniform bound on $I_t(x,y)$ for $|y-x| \ge 1$ stated in \eqref{e:qcb2}.

It remains to prove the integrability in \eqref{e:qcb1}. Let $L_I$ denote the linear span of the basis directions $\{e_i\}_{i \in I}$ equipped with its Lebesgue measure $dv_I$; if $I$ is empty then we take $L_I$ as the origin and $dv_I$ a Delta mass at the origin. By \eqref{e:itbound}, Lemmas \ref{l:numbound} and \ref{l:denombound}, for every $t \in [0,1)$ and $x \in F_1$ we have
\begin{align}
 \nonumber\int_{\substack{y \in F_2 \\ |y-x| \le 1 }} \!  I_t(x,y) \, dv_{F_y}(y) & \le c \int_{\substack{y \in F_2 \\ |y-x| \le 1 }} \! \big( \max\{ (1-t)^{1/2}, |y-x| \}^{2 \min\{d_1,d_2\} } (1-t) \big)^{-1/2} \, dv_{F_y}(y)  \\
 \label{e:itbound2} &  \qquad \le c \int_{\substack{u \in L_{I_2} \\  |u| \le 1 }} \!  \max\{ (1-t)^{1/2}, |u| \}^{-d_2} (1-t)^{-1/2}   \, dv_{I_2}(u) 
 \end{align}
where in the last inequality we used the triangle inequality and the fact that the integrand is non-increasing in $|y-x|$. It remains to show that \eqref{e:itbound2} is integrable over $t \in [0,1)$. For this we first make the substitution $t \mapsto 1-t$, and then split the integral depending on whether $|u| \le t^{1/2}$ or $|u| \ge t^{1/2}$. In the first case we have
\[   \int_{  |u| \le t^{1/2}  } \! \max\{ t^{1/2}, |u| \}^{-d_2} t^{-1/2}  \, dv_{I_2}(u)  =   t^{-d_2/2-1/2}  \int_{|u| \le t^{1/2}}  \, dv_{I_2}(u)   = c  t^{- 1/2 } , \]
whereas in the second case we have
\begin{align*}
\int_{  |u| \ge t^{1/2}  } \! \max\{ t^{1/2}, |u| \}^{-d_2} t^{-1/2}  \, dv_{I_2}(u)  & = t^{-1/2} \int_{  |u| \ge t^{1/2}  } \! |u|^{-d_2}   \, dv_{I_2}(u)  \\
& = \begin{cases}   ct^{-1/2} \int_{t^{1/2}}^1 \! r^{-1} \, dr  &  \text{if } d_2 \ge 1 \\ ct^{-1/2} &  \text{if }  d_2 = 0 \end{cases}
 \\
 &\le   c     t^{-1/2} \log(1/t)  ,
 \end{align*}
 where we integrated in polar coordinates in the second equality in the case $d_2 \ge 1$. Since $  t^{-1/2} \log(1/t) $ is integrable over $t \in [0,1]$, we are done.
\end{proof}

\begin{proof}[Proof of Proposition \ref{p:qcb2}]
The proof is identical to that of Proposition \ref{p:qcb}, except that we apply the general case $k \in \N$ of Lemma \ref{l:denombound}.
\end{proof}

\subsection{Proof of the pointwise bounds}

We collect some properties of the $\DC$ operator:
 
\begin{lemma}
\label{l:dc}
Let $(X,Y)$ be a Gaussian vector of dimension $m_1+m_2$.
\begin{enumerate}
    \item For $A \in \R^{m_1 \times m_1}$, 
\[  \DC(AX) = \det(A)^2\DC(X). \]
In particular, for $a \in \R$ and $B \in \R^{m_1 \times m_2}$,
\[ \DC(aX,Y)=a^{2m_1}\DC(X,Y) \quad \text{and} \quad  \DC(X+BY,Y)=\DC(X,Y) . \]
    \item If $X_n$ is a sequence converging to $X$ in $L^2$, then $\DC(X_n)\to\DC(X)$.
     \item If $X$ is non-degenerate then $\DC(Y) \ge \DC(Y|X) = \DC(Y,X)/ \DC(X)  $.
    \item  Let $\widetilde{Y}$ be an independent copy of $Y$, and define $Y^t = t Y + \sqrt{1 - t^2} \widetilde{Y} $, $t \in [0,1]$. If $X$ is non-degenerate, then $\DC(X,Y^t) \ge  \DC(X,Y^1) = \DC(X,Y) $.
\end{enumerate}
\end{lemma}
\begin{remark}
    In the proof of this lemma and below, we write $\Cov[Y|X]$ to denote the covariance matrix of $Y$ conditioned on $X=x$ for an arbitrary $x\in\R^{m_1}$. By Gaussian regression (Section~\ref{ss:GaussianRegression}) the matrix does not depend on the value of $x$, which justifies the choice of notation.
\end{remark}
\begin{proof}
The first item is elementary, the second follows from the continuity of the determinant operator.

Let $\Sigma_X$ and $\Sigma_Y$ be the covariance matrices of $X$ and $Y$ respectively, and $\Sigma_{YX}$ the cross-covariance matrix. By Gaussian regression (Section~\ref{ss:GaussianRegression})
\begin{displaymath}
    \Cov[Y|X]=\Sigma_Y-\Sigma_{YX}\Sigma_X^{-1}\Sigma_{YX}^T.
\end{displaymath}
Since the determinant is non-decreasing in the positive definite ordering, the determinant of the above expression is at most $\DC(Y)$. The identity
\begin{displaymath}
    \begin{pmatrix}
    \Sigma_Y &\Sigma_{YX}\\
     \Sigma_{YX}^T&\Sigma_X
    \end{pmatrix}=\begin{pmatrix}
        I &\Sigma_{YX}\\
        0 &\Sigma_X
    \end{pmatrix}
    \begin{pmatrix}
        \Sigma_Y-\Sigma_{YX}\Sigma_X^{-1}\Sigma_{YX}^T &0\\
        \Sigma_X^{-1}\Sigma_{YX}^T & I
    \end{pmatrix}
\end{displaymath}
proves that $\DC(Y,X)=\DC(X)\DC(Y|X)$, verifying the third item.

By the definition of $Y^t$ and the analogue of the previous identity
\begin{displaymath}
    \DC(Y^t,X)=\DC\begin{pmatrix}
        \Sigma_Y &t\Sigma_{YX}\\
        t\Sigma_{YX}^T &\Sigma_X
    \end{pmatrix}=\det(\Sigma_X) \det \big(  \Sigma_Y - t^2 \Sigma_{YX} \Sigma_X^{-1} \Sigma_{YX}^T \big).
\end{displaymath}
Since the determinant is non-decreasing in the positive-definite ordering, 
\[ t \mapsto \det \big(  \Sigma_Y -t^2 \Sigma_{YX} \Sigma_X^{-1} \Sigma_{YX}^T \big) \]
is non-increasing on $t \in [0,1]$, and hence so is $t \mapsto \DC(Y^t,X)$.
\end{proof}

Let us establish two particular consequences of Lemma \ref{l:dc}. These can be viewed as examples of the `divided difference' method (see \cite[Section 4.2]{bmm23} for an overview).

\begin{lemma}
\label{l:dd}
Suppose $f$ is a $C^3$-smooth stationary Gaussian field, and suppose that, for every $I_1, I_2 \subseteq \{1,2,\ldots, d\}$ with $I_3 = I_1 \cap I_2$ non-empty, and every unit vector $v \in \mathbb{S}^{d-1}$, the Gaussian vector 
\[ \big(   \partial_v \nabla_{I_3} f(0) , f(0),\nabla_{I_1 \cup I_2} f(0)  \big) \]
is non-degenerate, where $\partial_v f$ denotes the partial derivative in the direction $v$.  Then
\[ \DC \Big( \frac{ \nabla_{I_3} f(x) - \nabla_{I_3} f(0) }{|x|} \, \Big| \, f(0),\nabla_{I_1} f(0), \nabla_{I_2 \setminus I_3} f(x) \Big) \]
is uniformly bounded below as $x \to 0$ (with $x \neq 0$).
\end{lemma}

\begin{proof}
Suppose for contradiction that there exists a sequence $x_n\to 0$ such that
\begin{equation}\label{e:dd0}
    \DC \Big( \frac{ \nabla_{I_3} f(x_n) - \nabla_{I_3} f(0) }{|x_n|} \, \Big| \, f(0),\nabla_{I_1} f(0), \nabla_{I_2 \setminus I_3} f(x_n) \Big)\to 0.
\end{equation}
By the third item of Lemma~\ref{l:dc} this determinant is equal to
\begin{equation}\label{e:dd1}
\DC \Big( \frac{ \nabla_{I_3} f(x_n) - \nabla_{I_3} f(0) }{|x_n|} , f(0),\nabla_{I_1} f(0), \nabla_{I_2 \setminus I_3} f(x_n) \Big)\Big\slash\DC \big( f(0),\nabla_{I_1} f(0), \nabla_{I_2 \setminus I_3} f(x_n) \big).
\end{equation}
Since $f\in C^3$, $\nabla_{I_2\setminus I_3}f(x_n)\to\nabla_{I_2\setminus I_3}f(0)$ almost surely as $n\to\infty$. Since these variables are jointly Gaussian, convergence occurs in $L^2$ and so by the second item of Lemma~\ref{l:dc}
\begin{equation}\label{e:dd2}
    \DC \Big( f(0),\nabla_{I_1} f(0), \nabla_{I_2 \setminus I_3} f(x_n) \Big)\to\DC \Big( f(0),\nabla_{I_1} f(0), \nabla_{I_2 \setminus I_3} f(0) \Big)>0
\end{equation}
where the final inequality follows by assumption.

By passing to a subsequence, we may assume that $\hat{x}_n:=x_n/\lvert x_n\rvert\to v$ for some $v\in\mathbb{S}^{d-1}$. By a Taylor expansion, for $\lvert x_n\rvert\leq 1$
\begin{displaymath}
    \left\lvert\frac{ \nabla_{I_3} f(x_n) - \nabla_{I_3} f(0) }{|x_n|}-\partial_{\hat{x}_n}\nabla_{I_3}f(0)\right\rvert\leq c\lvert x_n\rvert\cdot\|f\|_{C^3(B(1))}
\end{displaymath}
where $c>0$ depends only on the dimension $d$, and $B(1)$ denotes the Euclidean ball of unit radius. Therefore
\begin{displaymath}
    \frac{ \nabla_{I_3} f(x_n) - \nabla_{I_3} f(0) }{|x_n|}\to\partial_v\nabla_{I_3}f(0)
\end{displaymath}
almost surely, and hence in $L^2$. By the second item of Lemma~\ref{l:dc} the numerator of \eqref{e:dd1} converges to
\begin{displaymath}
    \DC \Big( \partial_v\nabla_{I_3} f(0), f(0),\nabla_{I_1} f(0), \nabla_{I_2 \setminus I_3} f(0) \Big)
\end{displaymath}
which is positive by assumption. Together with \eqref{e:dd2}, this contradicts \eqref{e:dd0}.
\end{proof}

\begin{lemma}
\label{l:dd2}
Suppose $f$ is a $C^{3}$-smooth stationary Gaussian field such that, for every unit vector $v\in\mathbb{S}^{d-1}$, the Gaussian vector 
\[ \big( f(0),\nabla f(0),\partial_v\nabla f(0),\partial_v^3f(0) \big) \]
is non-degenerate. Then 
\[ \Var \big[ f(0) \, \big| \, f^t(x) - f(0),  \nabla f(0) , \nabla f^t(x) ]  \]
is uniformly bounded below as $(x,t) \to (0,1)$ (with $x\neq 0$ and $t\neq 1$).
\end{lemma}
\begin{proof}
Suppose for contradiction that the conditional variance is not bounded below, so there exists a sequence $(x_n,t_n)\to (0,1)$ such that, by the third item of Lemma~\ref{l:dc},
\begin{equation}\label{e:dd3}
    \frac{\DC (f(0),f^{t_n}(x_n)-f(0),\nabla f(0),\nabla f^{t_n}(x_n))}{\DC (f^{t_n}(x_n)-f(0),\nabla f(0),\nabla f^{t_n}(x_n))}\to 0.
\end{equation}
By passing to a subsequence, we may assume that
\begin{displaymath}
 \hat{x}_n:=\frac{x_n}{\lvert x_n\rvert}\to v\in \mathbb{S}^{d-1},\quad \frac{\lvert x_n\rvert}{\sqrt{1-t_n^2}+\lvert x_n\rvert}\to\alpha\in[0,1],\quad\text{and}\quad\frac{\lvert x_n\rvert^2}{\sqrt{1-t_n^2}+\lvert x_n\rvert^2}\to\beta\in[0,1].
\end{displaymath}
By the first item of Lemma~\ref{l:dc}, the left hand side of \eqref{e:dd3} is unchanged if we replace $\nabla f^{t_n}(x_n)$ in the numerator and denominator by
\begin{displaymath}
    A_n:=\frac{\nabla f^{t_n}(x_n)-t_n\nabla f(0)}{\sqrt{1-t_n^2}+\lvert x_n\rvert}.
\end{displaymath}
By Taylor's theorem, for $\lvert x_n\rvert\leq 1$,
\begin{displaymath}
    A_n=\frac{\sqrt{1-t_n^2}\nabla\Tilde{f}(x_n)+t_n\lvert x_n\rvert\partial_{\hat{x}_n}\nabla f(0)}{\sqrt{1-t_n^2}+\lvert x_n\rvert}+O_{C^{3}(B(1))}\left(\lvert x_n\rvert\right)
\end{displaymath}
where we introduce the notation $U_n=V_n+O_{C^{3}(B(1))}(a_n)$ to mean that there exists a constant $c>0$ depending only on the dimension $d$ such that
\begin{displaymath}
    \lvert U_n-V_n\rvert\leq ca_n\left(\|f\|_{C^{3}(B(1))}+\|\Tilde{f}\|_{C^{3}(B(1))}\right).
\end{displaymath}
Therefore as $n\to\infty$
\begin{equation}\label{e:dd4}
    A_n\to (1-\alpha)\nabla\Tilde{f}(0)+\alpha\partial_v\nabla f(0)
\end{equation}
almost surely and hence in $L^2$.

We first consider the case that $\beta<1$. Again by the first item of Lemma~\ref{l:dc}, the left hand side of \eqref{e:dd3} is unchanged if we replace $f^t(x_n)-f(0)$ in the numerator and denominator by
\begin{displaymath}
    B_n:=\frac{f^{t_n}(x_n)-f(0)-t_nx_n\cdot\nabla f(0)}{\sqrt{1-t_n^2}+\lvert x_n\rvert^2}.
\end{displaymath}
By a Taylor expansion
\begin{align*}
    B_n&=\frac{\sqrt{1-t_n^2}\Tilde{f}(x_n)-(1-t_n)f(0)+t_n\frac{\lvert x_n\rvert^2}{2}\partial_{\hat{x}_n}^2f(0)}{\sqrt{1-t_n^2}+\lvert x_n\rvert^2}+O_{C^{3}(B(1))}(\lvert x_n\rvert)\to(1-\beta)\Tilde{f}(0)+\frac{\beta}{2}\partial_v^2f(0)
\end{align*}
almost surely and in $L^2$. Combining this with \eqref{e:dd4} and the second item of Lemma~\ref{l:dc}, we see that the numerator on the left hand side of \eqref{e:dd3} converges to
\begin{displaymath}
    \DC\Big(f(0),\nabla f(0),(1-\alpha)\nabla\Tilde{f}(0)+\alpha\partial_v\nabla f(0),(1-\beta)\Tilde{f}(0)+\frac{\beta}{2}\partial_v^2f(0)\Big).
\end{displaymath}
This expression is strictly positive since $\beta<1$ and the Gaussian vector $(f(0),\nabla f(0),\partial_v\nabla f(0))$ is non-degenerate (and independent of $\Tilde{f}$). The denominator of \eqref{e:dd3} converges to the same expression with $f(0)$ removed, which is also positive. Hence \eqref{e:dd3} has a strictly positive limit, yielding a contradiction.

Next we consider the case that $\beta=1$ (which implies $\alpha=1$). Note that the limiting $\DC()$ terms in the numerator and denominator considered above are both zero, so we must consider a higher order Taylor expansion. By passing to a further subsequence, we assume that
\begin{displaymath}
    \frac{\lvert x_n\rvert^3}{\sqrt{1-t_n^2}+\lvert x_n\rvert^3}\to\gamma\in[0,1].
\end{displaymath}
By the first item of Lemma~\ref{l:dc} we can replace $f^{t_n}(x_n)-f(0)$ in \eqref{e:dd3} by
\begin{displaymath}
    C_n:=\frac{f^{t_n}(x_n)-f(0)-t_nx_n\cdot\nabla f(0)-\frac{x_n}{2}\cdot\big(\nabla f^{t_n}(x_n)-t_n\nabla f(0)\big)}{\sqrt{1-t_n^2}+\lvert x_n\rvert^3}.
\end{displaymath}
By Taylor's theorem, there exists $\theta_n,\theta^\prime_n\in[0,1]$ such that
\begin{displaymath}
    f^{t_n}(x_n)-f(0)=\sqrt{1-t_n^2}\Tilde{f}(x_n)-(1-t_n)f(0)+t_n\lvert x_n\rvert\partial_{\hat{x}_n}f(0)+t_n\frac{\lvert x_n\rvert^2}{2}\partial_{\hat{x}_n}^2f(0)+t_n\frac{\lvert x_n\rvert^3}{3!}\partial_{\hat{x}_n}^3f(\theta_nx_n)
\end{displaymath}
and
\begin{align*}
x_n\cdot\big(\nabla f^{t_n}(x_n)-t_n\nabla f(0)\big)=&\sqrt{1-t_n^2}\lvert x_n\rvert\partial_{\hat{x}_n}\Tilde{f}(x_n)+t_n\lvert x_n\rvert^2\partial_{\hat{x}_n}^2f(0)+t_n\frac{\lvert x_n\rvert^3}{2}\partial_{\hat{x}_n}^3f(\theta^\prime_n x_n).
\end{align*}
We therefore see that
\begin{align*}
    C_n\to (1-\gamma)\Tilde{f}(0)-\frac{\gamma}{12}\partial_v^3f(0)
\end{align*}
almost surely and in $L^2$. Hence the numerator of \eqref{e:dd3} converges to
\begin{displaymath}
    \DC\Big(f(0),\nabla f(0),\partial_v\nabla f(0),(1-\gamma)\Tilde{f}(0)-\frac{\gamma}{12}\partial_v^3f(0)\Big)
\end{displaymath}
which is positive by assumption. The denominator of \eqref{e:dd3} converges to the same expression with $f(0)$ removed, which is also positive and therefore gives the required contradiction.

\end{proof}

\begin{proof}[Proof of Lemma \ref{l:numbound}]
Let $I_1, I_2 \subseteq \{1,2,\ldots, d\}$ be given, and abbreviate 
\[ F(t,x) = \DC \big( f(0),f^t(x),\nabla_{I_1} f(0), \nabla_{I_2} f^t(x)  \big) . \]
It suffices to prove the following two statements: 
 \begin{enumerate}
 \item For all $\delta \in (0,1]$ there is a $c > 0$ such that, if either $t \le 1 - \delta $ or $|x| \ge \delta$,
 \[ F(t,x) \ge c . \]
\item There are $\delta, c > 0$ such that, for all $t \ge 1-\delta$ and $|x| \le \delta$,
\[ F(t,x) \ge c \max\{(1-t)^{1/2}, |x| \}^{2d'} (1-t ). \]
\end{enumerate}
We prove these statements in turn. First note that, by the second item of Lemma \ref{l:dc} and the fact that $f$ is $C^1$-smooth, $F(t,x)$ is a continuous function of $(t,x) \in [0,1) \times \R^d$. Now for the first statement we note that, by the third item of Lemma \ref{l:dc},
\[ F(t,x)  = \DC \big( f(0), \nabla_{I_1} f(0) \big) \times \DC \big( f^t(x),  \nabla_{I_2} f^t(x) \, | \, f(0), \nabla_{I_1} f(0) \big) , \]
and by the first and third items of Lemma \ref{l:dc},
\begin{align*}
\DC \big( f^t(x),  \nabla_{I_2} f^t(x) \, | \, f(0), \nabla_{I_1} f(0) \big)  &\ge  \DC \big( f^t(x),  \nabla_{I_2} f^t(x) \, | \, f \big) \\
&=  \DC \big( (1-t^2)^{1/2} \tilde{f}(x),  (1- t^2)^{1/2} \nabla_{I_2} \tilde{f}(x) \big) . 
\end{align*}
By the first item of Lemma \ref{l:dc}, and also by stationarity and the equality in law of $f$ and $\tilde{f}$, the above is equal to 
\[ (1-t^2)^{ \textrm{dim}(I_2) + 1} \DC \big(  f(0),   \nabla_{I_2} f(0) \big)  .\]
Since $( f(0), \nabla f(0) )$ is non-degenerate, we conclude that there is a $c_1(\delta) > 0$ such that 
\begin{equation}
\label{e:Fbound1}
F(t,x) \ge c_1(\delta)
\end{equation}
for all $t \le 1-\delta$.

Moreover, by the fourth item of Lemma \ref{l:dc},
\[ F(t,x) \ge  F(1,x) = \DC \big(f(0),\nabla_{I_1} f(0), f(x), \nabla_{I_2} f(x) \big)   . \]
 By stationarity and since $\partial^\alpha K(x) \to 0$ as $|x| \to \infty$ for $|\alpha| \le 2$, 
 \[  \DC \big(f(0),\nabla_{I_1} f(0), f(x), \nabla_{I_2} f(x) \big)  \to \DC \big(f(0),\nabla_{I_1} f(0) \big)  \DC \big( f(0), \nabla_{I_2} f(0) \big)     \]
 as $|x| \to \infty$. By continuity of $F(t,x)$, and since $( f(0), \nabla f(0) )$ is non-degenerate, we conclude that there is a $c_2(\delta) > 0$ such that 
\begin{equation}
\label{e:Fbound2}
F(t,x) \ge  c_2(\delta)
\end{equation}
for all $|x| \ge \delta$. Combining \eqref{e:Fbound1} and \eqref{e:Fbound2} we have proven the first statement.

We turn to the second statement. Abbreviating $I_3 = I_1 \cap I_2$, by repeated application of the third item of Lemma~\ref{l:dc} we have
\begin{align*}
& F(t,x) = \DC \big( f(0),\nabla_{I_1} f(0), \nabla_{I_2 \setminus I_3} f^t(x) \big)   \times \DC \big( \nabla_{I_3} f^t(x) \,|\, f(0),\nabla_{I_1} f(0), \nabla_{I_2 \setminus I_3} f^t(x) \big)  \\
& \qquad \qquad \qquad\qquad\qquad\qquad \times \Var \big[ f^t(x) \,|\, f(0),\nabla_{I_1} f(0), \nabla_{I_2} f^t(x) \big]  =: F_1(t,x) \times F_2(t,x) \times F_3(t,x),
\end{align*} 
where we interpret $F_2(t,x) = 1$ if $I_3$ is empty. We make the following four claims, which combined give the second statement:
\begin{enumerate}
\item  $F_1(t,x)$ is uniformly bounded below as $x \to 0$;
\item $F_2(t,x) (1-t^2)^{-d'}$ is uniformly bounded below;
\item Restricting to $t \ge 1/2$, $ F_2(t,x) |x|^{-2 d'}$ is uniformly bounded below as $x \to 0$;
\item $F_3(t,x) (1-t^2)^{-1}$ is uniformly bounded below.
\end{enumerate}
Let us establish these claims in turn:
\begin{enumerate}
\item By the second and fourth items of Lemma \ref{l:dc}, as $x \to 0$,
\begin{align*}
 F_1(t,x) =  \DC \big( f(0),\nabla_{I_1} f(0), \nabla_{I_2 \setminus I_3} f^t(x) \big)  & \ge  \DC \big( f(0),\nabla_{I_1} f(0), \nabla_{I_2 \setminus I_3} f(x) \big) \to \DC \big( f(0),\nabla_{I_1 \cup I_2 } f(0) \big) ,
\end{align*}
and the claim follows since $(f(0),\nabla f(0))$ is non-degenerate.
\item We may assume that $d' = \text{dim}(I_3)  \ge 1$. By conditioning on $f$ and using the first and fourth items of Lemma \ref{l:dc},
\begin{align*}
F_2(t,x) & = \DC \big( \nabla_{I_3} f^t(x) \,|\, f(0),\nabla_{I_1} f(0), \nabla_{I_2 \setminus I_3} f^t(x) \big)  &  \\
& \ge  \DC \big((1-t^2)^{1/2} \nabla_{I_3} \tilde{f}(x) \, |\,  \nabla_{I_2 \setminus I_3} \tilde{f}(x) \big)  = (1-t^2)^{d'}  \DC  \big( \nabla_{I_3} f(0) \,| \, \nabla_{I_2 \setminus I_3} f(0) \big)  
\end{align*}
 where the second step also used stationarity and the equality in law of $f$ and $\tilde{f}$. The claim follows since  $\nabla f(0)$ is non-degenerate.
\item By conditioning on $\tilde{f}$ and by the first and fourth items of Lemma~\ref{l:dc}, 
\begin{align*}
F_2(t,x) & = \DC \big( \nabla_{I_3} f^t(x) \, | \, f(0),\nabla_{I_1} f(0), \nabla_{I_2 \setminus I_3} f^t(x) \big)  \\
& \ge \DC \big( t  \nabla_{I_3} f(x) \,| \, f(0),\nabla_{I_1} f(0), \nabla_{I_2 \setminus I_3} f(x) \big)  \\
&  \ge 2^{-2d'} |x|^{2d'} \DC \Big( \frac{ \nabla_{I_3} f(x) - \nabla_{I_3} f(0) }{|x|} \, \Big| \, f(0),\nabla_{I_1} f(0), \nabla_{I_2 \setminus I_3} f(x) \Big)
\end{align*}
where the last step used that $t  \ge 1/2$. Then the claim follows from Lemma \ref{l:dd} above.
\item By again conditioning on $f$ and using  the first and fourth items of Lemma \ref{l:dc},
\[ F_3(t,x) = \Var \big[f^t(x) \, | \, f(0),\nabla_{I_1} f(0), \nabla_{I_2} f^t(x) \big]  \ge (1-t^2) \Var \big[ f(0) \,| \, \nabla_{I_2} f(0) \big]  , \]
and the claim follows since  $(f(0), \nabla f(0) )$ is non-degenerate. \qedhere
\end{enumerate}
\end{proof} 

Before proving Lemma \ref{l:denombound} we establish a uniform bound on the mean of the derivatives of the field under the conditioning \eqref{e:cond3}:

\begin{lemma}
\label{l:unifbound}
There exists a $c > 0$ such that, for all $I_1,I_2 \subseteq \{1,2,\ldots,d\}$, $\ell \in \R$, $|\alpha| \le 2$, and $(t,x) \in [0,1) \times \R^d\setminus\{0\}$,
\[  \big| \E[ \partial^{\alpha} f(0) \, \big| \, f(0) = f^t(x) = \ell, \nabla_{I_1}f(0) = \nabla_{I_2}f^t(x) = 0 ]  \big| \le c |\ell| .\]
\end{lemma}

\begin{remark}
This result is only non-trivial in the case $\ell \neq 0$, since if $\ell = 0$ then in fact the left-hand side is zero by Gaussian regression.
\end{remark}

\begin{proof}[Proof of Lemma \ref{l:unifbound}]
By first conditioning on $\nabla_{I_1} f(0) = \nabla_{I_2} f^t(x) = 0$ and $f^t(x) - f(0) = 0$, and then on $f(0) = \ell$, we have by Gaussian regression (Section~\ref{ss:GaussianRegression}) that
\begin{align*}
&  \E[ \partial^\alpha f(0) \, \big| \, f(0) = f^t(x) = \ell, \nabla_{I_1}f(0) = \nabla_{I_2}f^t(x) = 0 ]   \\
 &  \qquad =  \E[ \partial^\alpha f(0) \, \big| \, f(0) =  \ell, f^t(x) - f(0) = 0, \nabla_{I_1}f(0) = \nabla_{I_2}f^t(x) = 0 ] =  \ell s_1  / s_2 ,
\end{align*} 
where 
\[ s_1 =  \Cov[ \partial^\alpha f(0), f(0) \, \big| \, f^t(x) - f(0), \nabla_{I_1} f(0) , \nabla_{I_2} f^t(x)  ]  \]
and
\[ s_2  = \Var \big[ f(0) \, \big| \,  f^t(x) - f(0), \nabla_{I_1} f(0) , \nabla_{I_2} f^t(x) \big] .\]
By the Cauchy-Schwarz inequality, and by the third item of Lemma \ref{l:dc}, 
\[ s_1 \le  \max\{ \Var[\partial^\alpha f(0)  ], \Var[f(0) ] \}  ,\] 
so it remains to show that $s_2$ is uniformly bounded away from zero over $(t,x) \in [0,1) \times \R^d\setminus\{0\}$.

By the third item of Lemma~\ref{l:dc}
\begin{equation}\label{e:unifbound}
    s_2\geq\Var\big[ f(0) \, \big| \,  f^t(x) - f(0), \nabla f(0) , \nabla f^t(x) \big] .
\end{equation}
Arguing as in the proof of Lemma~\ref{l:numbound}, since $\partial^\alpha K(x) \to 0$ as $|x| \to \infty$ we deduce that, for any $\delta > 0$, there is a $c > 0$ such that $s_2 \ge c$ if either $t \le 1-\delta$ or $|x| \ge \delta$. Lemma~\ref{l:dd2} and \eqref{e:unifbound} then show that $s_2$ is bounded away from zero for $(x,t)$ close to $(0,1)$, completing the proof.
\end{proof}

\begin{proof}[Proof of Lemma \ref{l:denombound}]
Expanding the determinants, applying H\"{o}lder's inequality, and by stationarity and the equality in law of $f$ and $\tilde{f}$, it suffices to prove that
\[   \E \big[ | \partial_v \partial_{v'} f(0)  |^k  \, \big|  \, f(0) =  f^t(x)  = \ell , \nabla_{I_1} f(0) =  \nabla_{I_2} f^t(x) = 0 \big] \le c (1 + |\ell|^k )    \]
for all $k \in \N$ and unit vectors $v,v' \in \mathbb{S}^{d-1}$, where the constant $c > 0$ depends only on the field. Since these are Gaussian variables, this moment bound is implied by uniform bounds on the mean
\[   \E \big[ | \partial_v \partial_{v'} f(0)  |  \, \big|  \, f(0) =  f^t(x)  = \ell , \nabla_{I_1} f(0) =  \nabla_{I_2} f^t(x) = 0 \big] \le c |\ell| , \]
which is a special case of Lemma \ref{l:unifbound} above, and on the variance
\[   \Var[  \partial_v \partial_{v'} f(0)    \, \big|  \, f(0) =  f^t(x)  = \ell , \nabla_{I_1} f(0) =  \nabla_{I_2} f^t(x) = 0 ] \le   \Var[  \partial_v \partial_{v'} f(0)   ] \le  c  \]
where in the first inequality we used the third item of Lemma \ref{l:dc}.
\end{proof}

\section{Topological arguments}
\label{s:top}

In this section we prove Lemmas~\ref{l:Topological_stability}, \ref{l:change} and \ref{l:AsympStab}. We start with a fundamental lemma of (stratified) Morse theory, which says roughly that if we smoothly perturb a function then the topology of its level set cannot change unless passing through a critical point. We will deduce Lemmas \ref{l:Topological_stability} and \ref{l:change} as consequences of this result.

Let $B$ be a stratified box (as defined in Section~\ref{s:stability}) with strata $\mathcal{F}$, let $A\subseteq B$, and let $I\subset\R$ be a compact interval. We say that a continuous map $H:A\times I\to B$ is a \emph{stratified isotopy of $A$} if, for each $t\in I$, $H(\cdot,t)$ is a homeomorphism onto its image and, for each $F\in\mathcal{F}$, $H(F\times\{t\})\subseteq F$. In the case that $A=B$, we also require that $H(B,t)=B$ for every $t\in I$.

\begin{lemma}
\label{l: morse continuity}
Let $B$ be a stratified box and $U$ an open neighbourhood of $B$. Let $g$ and $h$ be functions in class $C^2(U)$, and define $g_t=g+th$ for $t\in [0,1]$. Suppose that $g_t$ has no (stratified) critical points in $B$ at level $\ell$ for all $t\in [0,1]$. Then there exists a stratified isotopy $H:B\times[0,1]\to B$ such that, for all $t\in[0,1]$,
\begin{displaymath}
H(\{g=\ell\}\times\{t\})=\{g_t=\ell\}.
\end{displaymath}
In particular $N_\star(B,g_t,\ell)$ is constant over $t\in[0,1]$ for $\star\in\{\mathrm{ES},\mathrm{LS}\}$.
\end{lemma}
\begin{proof}
Let us consider the domain $B\times[0,1]$ and the function $G(x,t)=(g_t(x),t)$. The main idea of the lemma is that the $t$-sections of the preimage $G^{-1}(\ell,[0,1])$ are level sets of $g_t$ and this preimage is a trivial (stratified) fibration over the level set of $g$. Hence the (stratified) topology of level sets of $g_t$ does not change with $t$.

First, we claim that there exists a neighbourhood $W$ of $G^{-1}(\ell,[0,1])$ such that  the function $g_t$ has no stratified critical points in $W$. If this were not the case then by compactness we could find a sequence  $(x_n,t_n)$ such that $x_n$ is a stratified critical point of $g_t$ and $(x_n,t_n)\to (x,t)$ where $g_t(x)=\ell$.  By taking a further subsequence we may assume that all $x_n$ are in the same stratum $F$. It may not be the case that $F_x=F$ however we must have $F_x\subseteq\overline{F}$ (since $x_n\to x$) and therefore by continuity we have $\nabla_{F_x}g(x)=0$, which contradicts the fact that $\{g_t=\ell\}$ has no stratified critical points. Reducing $W$ further we can assume that $W$ is of the form $\{(x,t)\in B\times[0,1]: |g_t(x)-\ell|<\epsilon\}$ for some $\epsilon>0$. We can endow $W$ with a stratification inherited from the stratification of $B$. 

The absence of stratified critical points ensures that $G$ is a submersion when restricted to each stratum. It follows from the continuity of $g_t$ and boundedness of $W$ that $G$ is a proper map (i.e.\ the inverse image of any compact set is compact). Using these two facts, we may apply Thom's first isotopy lemma \cite[Section~8]{mat73} to claim that $W$ is a trivial fibration over $G^{-1}((\ell-\epsilon,\ell+\epsilon),0)$ and $G^{-1}(\ell,[0,1])$ is a trivial fibration over $G^{-1}(\ell,0)=g^{-1}(\ell)$. Moreover, these fibrations respect stratification. The proof of Thom's lemma is based on the construction of a flow $x_t$ which preserves the stratification, trivializes the fibration and along which $g_t(x_t)$ is constant. This immediately implies that there is a stratified isotopy between the level sets of the functions $g_t$. 
By the isotopy extension theorem \cite[Corollary~1.4]{edwards1971} we may extend this to an isotopy of $B$.
\end{proof}

\begin{proof}[Proof of Lemma~\ref{l:Topological_stability}]
    Since $g$ has no stratified critical points in $B$, there exists $\epsilon>0$ such that for all $x\in B$
    \begin{displaymath}
    \|\nabla_{F_x}g(x)\|+\lvert g(x)-\ell\rvert\geq\epsilon
    \end{displaymath}
    where we recall that $F_x$ denotes the stratum of $B$ containing $x$. If
    \begin{displaymath}
        \|g-g^\prime\|_{C^1(U)}<\epsilon,
    \end{displaymath}
    then for all $t\in[0,1]$ we see that $g+t(g^\prime-g)$ has no stratified critical points at level $\ell$. Hence by Lemma~\ref{l: morse continuity} we conclude that $N_\star(B,g,\ell)=N_\star(B,g^\prime,\ell)$, as required.
\end{proof}

\begin{proof}[Proof of Lemma \ref{l:change}]
We first consider the case that $x\in\partial B$. We choose a box $B^\prime\subset U$ which contains $x$ in its interior (so $B^\prime$ is not a subset of $B$). We assume $B^\prime$ is small enough that it intersects only one component of $\{g\geq\ell\}$ (i.e.\ the component containing $x$) and that $g$ has no stratified critical points at level $\ell$ in $B^\prime$ (this is possible since $x$ is non-degenerate). Let $h\in C^2_c(B^\prime)$ such that $h(x) > 0$. Then for any $\delta>0$
\begin{displaymath}
\{g+\delta h\geq\ell\}\cap B\setminus B^\prime=\{g-\delta h\geq\ell\}\cap B\setminus B^\prime.
\end{displaymath}
Therefore any change in the number of excursion/level set components in $B$ must be due to changes in $B^\prime\cap B$. Now choose $\delta>0$ sufficiently small that $g+th$ has no stratified critical points at level $\ell$ in $B^\prime$ for all $\lvert t\rvert\leq\delta$. By Lemma~\ref{l: morse continuity} there exists a stratified isotopy between the sets $\{g-\delta h\geq\ell\}\cap B^\prime$ and $\{g+\delta h\geq\ell\}\cap B^\prime$. In particular, both of these sets consist of a single component and have the same intersection with $\partial B^\prime$ (since $h\equiv 0$ on $\partial B^\prime$). This argument shows that the topology of the level set $\{g+th=\ell\}$ on a neighbourhood of $B$ does not change as $t$ varies in $[-\delta,\delta]$. Moreover only one component of this level set varies with $t$.

Consider the component of $\{g=\ell\}$ containing $x$; if this component intersects $\R^d\setminus B$ then the same will be true of the corresponding components of $\{g+th=\ell\}$ and so there will be no change in the number of interior components, i.e.\ $N_\star(B,g-\delta h,\ell)=N_\star(B,g+\delta h,\ell)$ (see the first two panels of Figure~\ref{fig:Topological}). Now assume that the level component containing $x$ does not intersect $\R^d\setminus B$. Then, since $B^\prime$ has no stratified critical points, the component of $\{g\geq\ell\}\cap B^\prime$ containing $x$ must intersect $\partial B$ only at $x$ (i.e.\ the excursion component touches the boundary at a single point). Since $h(x)>0$, the set $\{g+\delta h\geq\ell\}$ must contain a neighbourhood of $x$ whereas $\{g-\delta h\geq\ell\}$ does not contain $x$. Therefore the number of \emph{interior} components increases by one as we pass from $\{g+\delta h\geq\ell\}$ to $\{g-\delta h\geq\ell\}$ (see the bottom panel of Figure~\ref{fig:Topological}). Note that this does not depend on the choice of $h$. The same conclusion applies to the number of level set components, since these are the boundaries of the excursion sets.

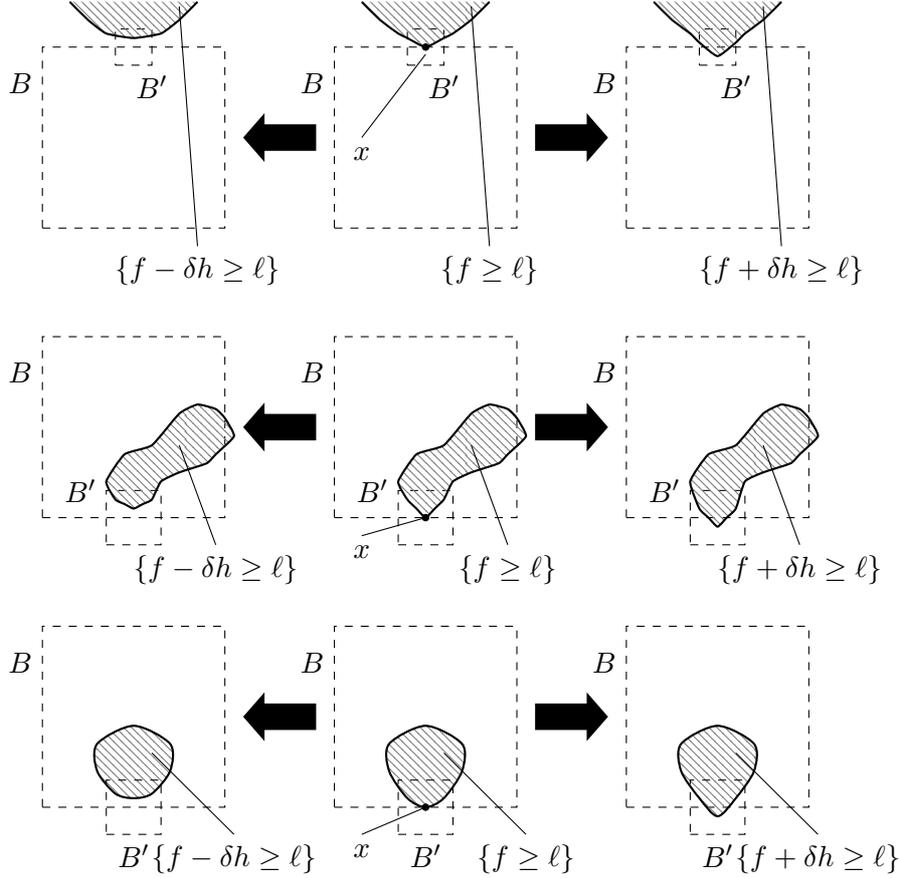
\begin{figure}[ht]
    \centering
\begin{tikzpicture}[scale=1.2]
\draw[dashed] (-1,-1) rectangle (1,1);
\node[left] at (-1,0.6) {$B$};
\draw[dashed] (-0.2,0.8) rectangle (0.2,1.2);
\node[below] at (0.2,0.8) {$B^\prime$};
\draw[thick,pattern=north west lines,pattern color=gray] plot [smooth,tension=0.3] coordinates {(-0.7,1.5)(-0.5,1.3)(-0.3,1.15)(-0.2,1.1)(0,1)(0.2,1.1)(0.3,1.15)(0.5,1.3)(0.7,1.5)};
\node[below] at (0.7,-1.2) {$\{f\geq\ell\}$};
\draw (0.7,-1.2)--(0.5,1.45);
\draw[fill] (0,1) circle (1pt);
\draw (0,0.92)--(-.7,0);
\node[below] at (-0.7,0) {$x$};

\draw[-{Triangle[width=18pt,length=8pt]}, line width=10pt](-1.2,0) -- (-2,0);
\draw[-{Triangle[width=18pt,length=8pt]}, line width=10pt](1.2,0) -- (2,0);

\begin{scope}[shift={(-3.2,0)}]

\draw[dashed] (-1,-1) rectangle (1,1);
\node[left] at (-1,0.6) {$B$};
\draw[dashed] (-0.2,0.8) rectangle (0.2,1.2);
\node[below] at (0.2,0.8) {$B^\prime$};
\draw[thick,pattern=north west lines,pattern color=gray] plot [smooth,tension=0.3] coordinates {(-0.7,1.5)(-0.5,1.3)(-0.3,1.15)(0,1.1)(0.3,1.15)(0.5,1.3)(0.7,1.5)};
\node[below] at (0.7,-1.2) {$\{f-\delta h\geq\ell\}$};
\draw (0.7,-1.2)--(0.5,1.45);

\end{scope}

\begin{scope}[shift={(3.2,0)}]
\draw[dashed] (-1,-1) rectangle (1,1);
\node[left] at (-1,0.6) {$B$};
\draw[dashed] (-0.2,0.8) rectangle (0.2,1.2);
\node[below] at (0.2,0.8) {$B^\prime$};
\draw[thick,pattern=north west lines,pattern color=gray] plot [smooth,tension=0.3] coordinates {(-0.7,1.5)(-0.5,1.3)(-0.3,1.15)(-0.2,1.05)(0,.9)(0.2,1.05)(0.3,1.15)(0.5,1.3)(0.7,1.5)};
\node[below] at (0.7,-1.2) {$\{f+\delta h\geq\ell\}$};
\draw (0.7,-1.2)--(0.5,1.45);

\end{scope}

\begin{scope}[shift={(0,-3.2)}]
\draw[dashed] (-1,-1) rectangle (1,1);
\node[left] at (-1,0.6) {$B$};
\draw[dashed] (-0.3,-1.3) rectangle (0.3,-0.7);
\node[left] at (-0.3,-0.7) {$B^\prime$};
\draw[thick,pattern=north west lines,pattern color=gray] plot [smooth cycle,tension=0.3] coordinates {(-0.3,-0.6)(-0.2,-0.8)(-0.1,-0.9)(0,-1)(0.1,-0.9)(0.2,-0.8)(0.3,-0.6)(0.5,-0.5)(0.8,-0.4)(0.9,-0.3)(1.1,-0.1)(1,0.1)(0.9,0.2)(0.7,0.25)(0.5,0.15)(0.2,-0.2)(-0.1,-0.3)};
\node[below] at (0.9,-1.3) {$\{f\geq\ell\}$};
\draw (0.9,-1.3)--(0.5,-0.2);
\draw[fill] (0,-1) circle (1pt);
\draw (0,-1)--(-.7,-1.2);
\node[below] at (-0.7,-1.2) {$x$};

\draw[-{Triangle[width=18pt,length=8pt]}, line width=10pt](-1.2,0) -- (-2,0);
\draw[-{Triangle[width=18pt,length=8pt]}, line width=10pt](1.2,0) -- (2,0);

\begin{scope}[shift={(-3.2,0)}]
\draw[dashed] (-1,-1) rectangle (1,1);
\node[left] at (-1,0.6) {$B$};
\draw[dashed] (-0.3,-1.3) rectangle (0.3,-0.7);
\node[left] at (-0.3,-0.7) {$B^\prime$};
\draw[thick,pattern=north west lines,pattern color=gray] plot [smooth cycle,tension=0.3] coordinates {(-0.3,-0.6)(-0.2,-0.8)(-0.1,-0.85)(0,-0.9)(0.1,-0.85)(0.2,-0.8)(0.3,-0.6)(0.5,-0.5)(0.8,-0.4)(0.9,-0.3)(1.1,-0.1)(1,0.1)(0.9,0.2)(0.7,0.25)(0.5,0.15)(0.2,-0.2)(-0.1,-0.3)};
\node[below] at (0.9,-1.3) {$\{f-\delta h\geq\ell\}$};
\draw (0.9,-1.3)--(0.5,-0.2);

\end{scope}

\begin{scope}[shift={(3.2,0)}]
\draw[dashed] (-1,-1) rectangle (1,1);
\node[left] at (-1,0.6) {$B$};
\draw[dashed] (-0.3,-1.3) rectangle (0.3,-0.7);
\node[left] at (-0.3,-0.7) {$B^\prime$};
\draw[thick,pattern=north west lines,pattern color=gray] plot [smooth cycle,tension=0.3] coordinates {(-0.3,-0.6)(-0.2,-0.9)(-0.1,-1)(0,-1.1)(0.1,-1)(0.2,-0.9)(0.3,-0.6)(0.5,-0.5)(0.8,-0.4)(0.9,-0.3)(1.1,-0.1)(1,0.1)(0.9,0.2)(0.7,0.25)(0.5,0.15)(0.2,-0.2)(-0.1,-0.3)};
\node[below] at (0.9,-1.3) {$\{f+\delta h\geq\ell\}$};
\draw (0.9,-1.3)--(0.5,-0.2);

\end{scope}
\end{scope}

\begin{scope}[shift={(0,-6.4)}]
\draw[dashed] (-1,-1) rectangle (1,1);
\node[left] at (-1,0.6) {$B$};
\draw[thick,pattern=north west lines,pattern color=gray] plot [smooth cycle,tension=0.5] coordinates {(0,-1)(.2,-0.9)(0.4,-0.6)(0.4,-0.3)(0,-0.1)(-0.4,-0.3)(-0.4,-0.6)(-0.2,-0.9)};
\draw[dashed] (-0.3,-1.3) rectangle (0.3,-0.7);
\node[below] at (0,-1.3) {$B^\prime$};
\node[below] at (1.1,-1.3) {$\{f\geq\ell\}$};
\draw (1.1,-1.3)--(0.2,-0.4);
\draw[fill] (0,-1) circle (1pt);
\draw (0,-1)--(-.7,-1.3);
\node[below] at (-0.7,-1.3) {$x$};

\draw[-{Triangle[width=18pt,length=8pt]}, line width=10pt](-1.2,0) -- (-2,0);
\draw[-{Triangle[width=18pt,length=8pt]}, line width=10pt](1.2,0) -- (2,0);

\begin{scope}[shift={(-3.2,0)}]
\draw[dashed] (-1,-1) rectangle (1,1);
\node[left] at (-1,0.6) {$B$};
\draw[thick,pattern=north west lines,pattern color=gray] plot [smooth cycle,tension=0.5] coordinates {(0,-0.9)(.2,-0.85)(0.4,-0.6)(0.4,-0.3)(0,-0.1)(-0.4,-0.3)(-0.4,-0.6)(-0.2,-0.85)};
\draw[dashed] (-0.3,-1.3) rectangle (0.3,-0.7);
\node[below] at (0,-1.3) {$B^\prime$};
\node[below] at (1.1,-1.3) {$\{f-\delta h\geq\ell\}$};
\draw (1.1,-1.3)--(0.2,-0.4);
\end{scope}

\begin{scope}[shift={(3.2,0)}]
\draw[dashed] (-1,-1) rectangle (1,1);
\node[left] at (-1,0.6) {$B$};
\draw[thick,pattern=north west lines,pattern color=gray] plot [smooth cycle,tension=0.5] coordinates {(0,-1.1)(.2,-0.9)(0.4,-0.6)(0.4,-0.3)(0,-0.1)(-0.4,-0.3)(-0.4,-0.6)(-0.2,-0.9)};
\draw[dashed] (-0.3,-1.3) rectangle (0.3,-0.7);
\node[below] at (0,-1.3) {$B^\prime$};
\node[below] at (1.1,-1.3) {$\{f+\delta h\geq\ell\}$};
\draw (1.1,-1.3)--(0.2,-0.4);
\end{scope}
\end{scope}

\end{tikzpicture}
\caption{Stratified critical points on the boundary and the effect of local perturbations on excursion/level sets. The excursion sets $\{f\geq\ell\}$ in the top and middle panels intersect $\R^d\setminus B$ and so they do not contribute to the (interior) component count regardless of the local perturbation around $x$. The excursion set $\{f\geq\ell\}$ in the bottom panel is interior but touches the boundary at $x$, therefore the component count is changed by a local perturbation.}
\label{fig:Topological}
\end{figure}

Now let us assume that $x$ is in the interior of $B$. Let $B^\prime\subset B$ be a box which intersects only one component of $\{g\geq\ell\}\cap B$ and contains $x$ in its interior. Let $h\in C^2_c(B^\prime)$ with $h(x)>0$ and let $V$ be an open ball around $x$ on which $h>0$. Then for $t\in\R$
\begin{equation}\label{e:Morse1}
    \{g+th\geq\ell\}\cap V=\{G\geq t\}\cap V
\end{equation}
where $G:=(g-\ell)/h$. Moreover since $\nabla G=h^{-1}\nabla g-h^{-2}(g-\ell)\nabla h$ we see that $x$ is a critical point for $G$ at level $0$. Calculating the Hessian of $G$ in a similar way shows that $x$ is a non-degenerate critical point for $G$. By the Morse lemma \cite[Theorem 1.1.12]{nic11} (and reducing the size of $V$ if necessary) there is a $C^2$-chart $\phi:V\to \R^d$ such that $\phi(x)=0$ and 
\begin{equation}\label{e:Morse2}
    G\circ\phi^{-1}(y_1,\dots,y_d)=-y_1^2-\dots-y_{i}^2+y_{i+1}^2+\dots+y_d^2
\end{equation}
where $i$ is the index of the critical point $x$. (Note that this index depends on $g$ but not on $h$.)

We consider how the excursion sets of $g+th$ vary on each of the sets $V$, $B^\prime\setminus V$ and $B\setminus B^\prime$. Let $\delta>0$ be small enough that $g+th$ has no stratified critical points at level $\ell$ in $B^\prime\setminus V$ for all $\lvert t\rvert\leq \delta$. By Lemma~\ref{l: morse continuity} there exists a stratified isotopy between the sets
\begin{displaymath}
    \{g+\delta h\geq\ell\}\cap B^\prime\setminus V\quad\text{and}\quad \{g-\delta h\geq\ell\}\cap B^\prime \setminus V.
\end{displaymath}
(Strictly speaking we must repeat the proof of Lemma~\ref{l: morse continuity} for $B^\prime\setminus V$ since this is not a box, however the arguments are identical.) Since $h$ is supported on $B^\prime$,
\begin{displaymath}
    \{g+\delta h\geq\ell\}\cap B\setminus B^{\prime}=\{g-\delta h\geq\ell\}\cap B\setminus B^{\prime}.
\end{displaymath}
Therefore any change in the topology of excursion/level sets must be captured in $V$.

By \eqref{e:Morse1} and \eqref{e:Morse2} we see that if $i=0$ then as $t$ moves from $-\delta$ to $\delta$ the excursion set changes from the empty set to a ball (and the reverse occurs if $i=n$). In either case the number of excursion/level sets either increases or decreases by one. If $i\in\{1,n-1\}$, then the level set goes from a one-sheet hyperboloid to two-sheet hyperboloid (with a cone at $t=0$). In this case we gain (or lose) one level set component in $V$. The number of excursion set components might be constant (if $\lambda=n-1$) or change by one (if $\lambda=1$). In all other cases (i.e. $1<\lambda<n-1$) the corresponding hyperbolic surface is connected for all $t$ and separates $\R^n$ into two domains. This means that the number of components does not change.  

Note that the above argument depends only on the canonical form, which is independent of $h$, hence the change in the number of components is the same for all functions $h$.
\end{proof}

\begin{proof}[Proof of Lemma~\ref{l:AsympStab}]
Let $A$ be the component of $\{g\geq \ell\}$ containing $x$ and let $A_1,\dots, A_n$ be the components of $A\setminus\{x\}$. (The Morse lemma \cite[Theorem~1.1.12]{nic11} implies that there are only finitely many such components.) By reordering, we may assume that $A_1,\dots, A_m$ are bounded and $A_{m+1},\dots,A_n$ are unbounded for some $m\in\{0,\dots,n\}$. 

Now fix an $r>1$ such that $A_1,\dots,A_m\subset x+\Lambda_r$ (if $m=0$ then simply choose $r=1$). We claim that $r$ satisfies the conclusion of the lemma, i.e.\ there exists a $\Delta\in\{+,-,0\}$ such that for any box $B\supset x+\Lambda_r$ the point $x$ is $(\Delta,B)$-pivotal. By Lemma~\ref{l:change}, there exists $\Delta\in\{+,-,0\}$ such that $x$ is $(\Delta,x+\Lambda_r)$-pivotal. By definition, for any sufficiently small neighbourhood $W$ of $x$, any function $h\in C^2_c(W)$ satisfying $h(x)>0$ and any $\delta$ sufficiently small
\begin{equation}\label{e:AsympStab1}
    N_{\star}(x+\Lambda_r,g+\delta h,\ell)-N_{\star}(x+\Lambda_r,g-\delta h,\ell)=\begin{cases}
        1 &\text{if }\Delta=+\\
        -1 &\text{if }\Delta=-\\
        0 &\text{if }\Delta=0
    \end{cases}.
\end{equation}
If $B\supset(x+\Lambda_r)$ then by definition
\begin{equation}\label{e:AsympStab2}
    N_\star(B,g-\delta h,\ell)=N_\star(x+\Lambda_r,g-\delta h,\ell)+\tilde{N}_\star(B,x+\Lambda_r,g-\delta h,\ell)
\end{equation}
where the final term denotes the number of components of $\{g-\delta h\geq\ell\}$ which are contained in $B$ and intersect $B\setminus(x+\Lambda_r)$. Any component contributing to this latter term must be bounded and is not contained in $x+\Lambda_r$, so cannot intersect $A_1\cup\dots\cup A_m$. Therefore there is a neighbourhood $W^\prime$ of $x$ which does not intersect any component contributing to $\tilde{N}_\star(B,x+\Lambda_r,g-\delta h,\ell)$. These components are disjoint from perturbations on $W^\prime$ and therefore
\begin{displaymath}
    \tilde{N}_\star(B,x+\Lambda_r,g-\delta h,\ell)=\tilde{N}_\star(B,x+\Lambda_r,g+\delta h,\ell)
\end{displaymath}
for $h\in C^2_c(W^\prime)$. Combining this with \eqref{e:AsympStab1}, \eqref{e:AsympStab2}, and the corresponding equation with $g+\delta h$ replacing $g-\delta h$ we have
\begin{displaymath}
    N_{\star}(B,g-\delta h,\ell)-N_{\star}(B,g+\delta h,\ell)=\begin{cases}
        1 &\text{if }\Delta=+\\
        -1 &\text{if }\Delta=-\\
        0 &\text{if }\Delta=0
    \end{cases}
\end{displaymath}
completing the proof.
\end{proof}

\begin{appendix}
\section{Non-degeneracy properties of smooth Gaussian fields}
In this appendix we collect some basic non-degeneracy properties of smooth Gaussian fields. Throughout we fix $k \in \mathbb{N}$ and let $f$ be a $C^k$-smooth Gaussian field on $\R^d$. We define $\mathcal{A}$ to be the collection of multi-indices on $d$ variables with order at most $k$, and for $A \subseteq \mathcal{A}$ we let \[ \textrm{deg} (A) = \max\{ |\alpha| : \alpha \in A\} \le k \]
be the largest order among indices in $A$. For $x \in \R^d$ and $A \subseteq \mathcal{A}$, we define the Gaussian vector $D^A f(x) = \big(   \partial^\alpha f(x)  \big)_{\alpha \in A}$.

\subsection{Non-degeneracy of the field and its derivatives}
Our first lemma concerns non-degeneracy properties of $D^A f(x)$, including when jointly evaluated at multiple points.

We say that a set of multi-indices $A \subseteq \mathcal{A}$ is \textit{spherical} if there exists a multi-index $\alpha' \in A$ such that
\[ \big\{ \alpha' + \alpha : \alpha \in \{(0,\dots,0), (2,0,\ldots,0) , (0,2,\ldots ,0) , \ldots , (0,0,\ldots,2) \} \big\} \subseteq A ,\]
and \textit{non-spherical} otherwise.
For example, if $d=k=3$ then
\[   \{ (1,0,0), (3,0,0), (1,2,0), (1,0,2) \} \quad \text{and} \quad \{ (1,0,1), (3,0,0), (1,2,0), (1,0,2)  \} \]
are respectively spherical and non-spherical. 

 \begin{lemma}
\label{l:nondegen1}
Suppose $f$ is stationary and let $\mu$ be its spectral measure. Let $x_1, \ldots, x_n \in \R^d$ be distinct points, and let $A_1, \ldots, A_n \subseteq \mathcal{A}$.
\begin{enumerate}
\item If the support of $\mu$ contains an open set, the Gaussian vector $(D^{A_i} f(x_i))_{1 \le i \le n}$ is non-degenerate.
\item Suppose each $A_i$ is non-spherical, and either (i) $n=1$, or (ii) $n=2$ and $\textrm{deg}(A_i) \le 1$, $i=1,2$. Then if the support of $\mu$ contains a scaled sphere $a \mathbb{S}^{d-1}$, $a >0$, the Gaussian vector $(D^{A_i} f(x_i))_{1 \le i \le n}$ is non-degenerate.
\end{enumerate}
\end{lemma}

\begin{remark}
\label{r:nondegen}
In the paper we do not use the full strength of Lemma \ref{l:nondegen1}. Indeed we only use the fact that if the support of $\mu$ contains either (i) an open set, or (ii) $a \mathbb{S}^{d-1}$, then the following vectors are non-degenerate for $x \neq 0$ and $v \in \mathbb{S}^{d-1}$:
\begin{itemize}
    \item $(f(0),\nabla f(0), f(x), \nabla f(x)) $;
    \item $ (f(0), \nabla f(0), \partial_v \nabla f(0), \partial^3_v f(0) ) $;
     \item (only in case (i)) $(f(0),\nabla f(0), \nabla^2 f(0) )$.
\end{itemize}
Note in particular that the second vector can be written (after a linear transformation of full rank) as $D^A f(0)$ for a non-spherical $A \subset \mathcal{A}$.
\end{remark}

\begin{remark}
We conjecture that the restrictions on $n$ and $\textrm{deg}(A_i)$ in the second item of Lemma~\ref{l:nondegen1} are not necessary for the conclusion of the lemma, but their inclusion simplifies the proof and is sufficient for our needs. On the other hand, the restriction to non-spherical $A_i$ is necessary, for instance, if the support of $\mu$ is contained in a sphere $a \mathbb{S}^{d-1}$ then the vector $D^A f(0) = (f(0),\nabla^2 f(0))$, corresponding to a spherical $A \subset \mathcal{A}$, is degenerate.
\end{remark}

In the proof of Lemma \ref{l:nondegen1} we use some facts about the zero locus of certain analytic functions. For distinct $x_1, \ldots, x_n \in \R^d$, $A_1, \ldots, A_n \subseteq \mathcal{A}$, and real coefficients $c = (c_{i,\alpha})_{1 \le i \le n, \alpha \in A_i}$, define the function  
\[  g = g_{x_i;A_i;c} : \R^d  \to \mathbb{C}^d \ , \quad  g(t) = \sum_{1 \le i \le n,\alpha \in A_i} c_{i,\alpha} (-\mathrm{i} )^{\lvert\alpha\rvert}t^{\alpha}e^{-\mathrm{i} t\cdot x_i} , \]
and let $ Z = \{ t \in \R^d : g(t) = 0\}$ be its zero locus.

\begin{lemma}
\label{l:alggeo}
Fix distinct $x_1, \ldots, x_n \in \R^d$, $A_1, \ldots, A_n \subseteq \mathcal{A}$, and $c = (c_{i,\alpha})_{1 \le i \le n, \alpha \in A_i}$.
\begin{enumerate}
\item If $Z$ contains an open set, then $c=0$.
\item Suppose each $A_i$ is non-spherical, and either (i) $n=1$, $x_1 = 0$, or (ii) $n=2$, $x_1 = 0$, and $\deg(A_1), \deg(A_2) \le 1$. If $Z$ contains the sphere $\mathbb{S}^{d-1}$, then $c=0$.
\end{enumerate}
\end{lemma}
\begin{proof}
In both statements it suffices to prove that $g = 0$, since this implies that $c=0$ as follows. Suppose $c \neq 0$, and rewrite $g$ as
\begin{displaymath}
    g(t)=\sum_{\alpha'\in\mathcal{A}}(-\mathrm{i})^{\lvert\alpha'\rvert}t^{\alpha'} h_{\alpha'}(t) \ , \quad h_{\alpha'}(t) = \sum_{ \{1 \le i \le n, \alpha \in A_i : \alpha=\alpha' \} }  c_{i,\alpha} e^{- \textrm{i}t\cdot x_i}.
\end{displaymath}
Noting that $h_{\alpha'}(t)$ is bounded and non-zero almost everywhere, we may take $\lvert t\rvert\to\infty$ in such a way that $|g(t)| \to \infty$, and hence $g \neq 0$. So it remains to prove that $g$ is identically zero.

We separate $g$ into its real and imaginary parts
\[  g_1: \R^d \to \R^d \ , \quad g_1(t) = \sum_{i=1}^n\sum_{\substack{\alpha \in A_i \\ |\alpha| = 2j }} c_{i,\alpha} (-1)^j t^{\alpha} \cos(t\cdot x_i ) +\sum_{\substack{\alpha \in A_i \\ |\alpha| = 2j+1 }} c_{i,\alpha} (-1)^{j+1} t^{\alpha} \sin(t\cdot x_i )  \]  
and
\[ g_2: \R^d \to \R^d \ , \quad  g_2(t) = \sum_{i=1}^n\sum_{\substack{\alpha \in A_i \\ |\alpha| = 2j }} c_{i,\alpha} (-1)^{j+1} t^{\alpha} \sin(t\cdot x_i ) +\sum_{\substack{\alpha \in A_i \\ |\alpha| = 2j+1 }} c_{i,\alpha} (-1)^{j+1} t^{\alpha} \cos(t\cdot x_i )  ,  \]
both of whose zero locus contains $Z$.

For the first statement, since $g_1$ and $g_2$ are real-analytic functions whose zero loci contain open sets, $g_1$ and $g_2$ are identically zero, hence so is $g$. 

For the second statement, we split into the two cases (i) and (ii). 
In case (i), since $n=1$ and $x_1 = 0$, the functions
\[ P_1 : \mathbb{C}^d \to \mathbb{C} \ , \quad P_1(t)  = g_1(t)  = \sum_{\substack{\alpha \in A_1 \\ |\alpha| = 2j }} c_{1,\alpha} (-1)^{j} t^{\alpha} \]
and
\[ P_2 : \mathbb{C}^d \to \mathbb{C} \ , \quad P_2(t) = g_2(t) = \sum_{\substack{\alpha \in A_1 \\ |\alpha| = 2j+1}} c_{1,\alpha} (-1)^{j+1} t^{\alpha} \]
are polynomials that vanish on the real sphere $\mathbb{S}^{d-1}$, and hence also on the complex sphere since the former is Zariski dense in the latter. Therefore, by Hilbert's Nullstellensatz, both $P_1(t)$ and $P_2(t)$ have the polynomial $Q(t) = 1 - \sum_{1 \le i \le d}  t_i^2$ as a divisor, which contradicts the fact that $A_1$ is assumed non-spherical unless $g = 0$. In case (ii), without loss of generality we may assume $x_2$ lies on the $d^\textrm{th}$ coordinate axis. Then considering the restriction of $g_1$ and $g_2$ to $\{t_d = 0\}$, we see that
\[ P'_1 : \mathbb{C}^{d-1} \to \mathbb{C} \ , \quad P'_1(t) = g_1(t,0)  = c_{1,0} + c_{2,0}   \]
and
\[ P'_2 : \mathbb{C}^{d-1} \to \mathbb{C} \ , \quad P'_2(t) = g_2(t,0)=  (-1) \sum_{1 \le j \le d-1} (c_{1,j} + c_{2,j}) t_j    \]
are affine functions that vanish on the real sphere $\mathbb{S}^{d-2}$, and therefore have null coefficients. Thus viewed as a function on $\R^d$, $g(t)$ depends only on the coordinate $t_d$. Since it vanishes on $\mathbb{S}^{d-1}$, $Z$ contains an open set, and $g = 0$.
\end{proof}

\begin{proof}[Proof of Lemma \ref{l:nondegen1}] 
In the case $A_i = \{ (0,\ldots ,0) \}$ this is a well-known result (see \cite[Theorem 6.8]{wen05}), and Lemma \ref{l:alggeo} allows us to extend the argument to the general case.

Suppose for the sake of contradiction that the vector $(D^{A_i} f(x_i))_{1 \le i \le n}$ is degenerate. By definition this means that there exists non-zero $c = (c_{i,\alpha})_{1 \le i \le n, \alpha \in A_i}$ such that
\begin{align*}
0 = \mathrm{Var}\Big(\sum_{1 \le i \le n,\alpha \in A_i} c_{i,\alpha} \partial^{\alpha}f(x_i) \Big) = \int_{\R^d} \bigg\lvert \sum_{1 \le i \le n,\alpha \in A_i} c_{i,\alpha} (-\mathrm{i})^{\lvert\alpha\rvert}t^{\alpha}e^{-\mathrm{i}t\cdot x_i}\bigg\rvert^2 \; d\mu(t). 
\end{align*}
Therefore the function
\begin{displaymath}
    g(t) = \sum_{1 \le i \le n,\alpha \in A_i} c_{i,\alpha} (-\mathrm{i} )^{\lvert\alpha\rvert}t^{\alpha}e^{-\mathrm{i} t\cdot x_i}
\end{displaymath}
vanishes on the support of $\mu$, and the desired contradiction follows from Lemma~\ref{l:alggeo} (after transformation and scaling).
\end{proof}

\subsection{Non-degeneracy under the interpolation}
Let $\tilde{f}$ be an independent copy of the field $f$, and for $t \in [0,1]$ define $f^t = t f + \sqrt{1-t^2} \tilde{f}$. 

Our next lemma states that, for $t \in [0,1)$, the pair of fields $(f,f^t)$ inherits the non-degeneracy properties of $f$. In particular this shows the non-degeneracy of the conditioning \eqref{e:cond} in the definition of the pivotal measures in Theorem \ref{t:cf}.

\begin{lemma}
\label{l:nondegen2}
Let $x_1, x_2 \in \R^d$ be points, not necessarily distinct, and let $A_1, A_2 \subseteq \mathcal{A}$. Suppose the Gaussian vectors $D^{A_1} f(x_1)$ and $D^{A_2} f(x_1)$ are both non-degenerate. Then the Gaussian vector $(D^{A_1} f(x_1)  , D^{A_2} f^t(x_2))$ is non-degenerate for every $t \in [0,1)$.
\end{lemma}

\begin{proof}
By the third item of Lemma \ref{l:dc}
\begin{align*}
\textrm{DC} \big(D^{A_1} f(x_1), D^{A_2} f^t(x_2) \big) &=  \textrm{DC}\big(D^{A_1} f(x_1) \big) \textrm{DC} \big(  t D^{A_2} f(x_2) + \sqrt{1-t^2} D^{A_2} \tilde{f}(x_2) \, \big| \, D^{A_1} f(x_1)  \big)   \\
& \ge  \textrm{DC}\big(D^{A_1} f(x_1) \big)  \textrm{DC} \big(  t D^{A_2} f(x_2) + \sqrt{1-t^2} D^{A_2} \tilde{f}(x_2) \, \big| \, f   \big) .
\end{align*}
By the first item of Lemma \ref{l:dc}, and the equality in law of $f$ and $\tilde{f}$, the above is equal to
\begin{equation*}
 \textrm{DC}\big(D^{A_1} f(x_1) \big)  \textrm{DC} \big(  \sqrt{1-t^2} D^{A_2} \tilde{f} (x_2) \big)  =  (1-t^2)^{|A_2|}    \textrm{DC} \big(D^{A_1} f(x_1) \big) \textrm{DC} \big(  D^{A_2} f(x_2) \big) > 0  . 
\end{equation*}
\end{proof}

\end{appendix}

\medskip

\bigskip
\bibliographystyle{plain}
\bibliography{paper}

\begin{thebibliography}{10}

\bibitem{at07}
R.~J. Adler and J.~E. Taylor.
\newblock {\em Random fields and geometry}.
\newblock Springer Monographs in Mathematics. Springer, New York, 2007.

\bibitem{bbks86}
J.~M. Bardeen, J.~R. Bond, N.~Kaiser, and A.~S. Szalay.
\newblock The statistics of peaks of {G}aussian random fields.
\newblock {\em Astrophys. J.}, 304:15--61, 1986.

\bibitem{bg17}
V.~Beffara and D.~Gayet.
\newblock Percolation of random nodal lines.
\newblock {\em Publ. Math. Inst. Hautes \'{E}tudes Sci.}, 126:131--176, 2017.

\bibitem{bmm20}
D.~Beliaev, M.~McAuley, and S.~Muirhead.
\newblock On the number of excursion sets of planar {G}aussian fields.
\newblock {\em Probab. Theory Related Fields}, 178(3-4):655--698, 2020.

\bibitem{bmm20a}
D.~Beliaev, M.~McAuley, and S.~Muirhead.
\newblock Smoothness and monotonicity of the excursion set density of planar
  {G}aussian fields.
\newblock {\em Electron. J. Probab.}, 25:Paper No. 93, 37, 2020.

\bibitem{bmm22}
D.~Beliaev, M.~McAuley, and S.~Muirhead.
\newblock Fluctuations of the number of excursion sets of planar {G}aussian
  fields.
\newblock {\em Probab. Math. Phys.}, 3(1):105--144, 2022.

\bibitem{bmm23}
D.~Beliaev, M.~McAuley, and S.~Muirhead.
\newblock A central limit theorem for the number of excursion set components of
  {G}aussian fields.
\newblock {\em Ann. Probab.}, 52(3):882--922, 2024.

\bibitem{bmr20}
D.~Beliaev, S.~Muirhead, and A.~Rivera.
\newblock A covariance formula for topological events of smooth {G}aussian
  fields.
\newblock {\em Ann. Probab.}, 48(6):2845--2893, 2020.

\bibitem{bgt87}
N.~H. Bingham, C.~M. Goldie, and J.~L. Teugels.
\newblock {\em Regular variation}, volume~27 of {\em Encyclopedia of
  Mathematics and its Applications}.
\newblock Cambridge University Press, Cambridge, 1987.

\bibitem{bgh01}
S.~G. Bobkov, F.~G{\"{o}}tze, and C.~Houdr{\'{e}}.
\newblock On {G}aussian and {B}ernoulli covariance representations.
\newblock {\em Bernouilli}, 7(3):439--451, 2001.

\bibitem{bs02}
E.~Bogomolny and C.~Schmit.
\newblock Percolation model for nodal domains of chaotic wave functions.
\newblock {\em Phys. Rev. Lett.}, 88(11):114102, 2002.

\bibitem{cac82}
T.~Cacoullos.
\newblock On upper and lower bounds for the variance of a function of a random
  variable.
\newblock {\em Ann. Probab.}, 10(3):799--809, 1982.

\bibitem{cha08}
S.~Chatterjee.
\newblock Chaos, concentration, and multiple valleys.
\newblock {\em arXiv preprint, arXiv:0810.4221}, 2008.

\bibitem{cha14}
S.~Chatterjee.
\newblock {\em Superconcentration and related topics}.
\newblock Springer Monographs in Mathematics. Springer, 2014.

\bibitem{che82}
L.~H.~Y. Chen.
\newblock An inequality for the multivariate normal distribution.
\newblock {\em J. Multivariate Anal.}, 12(2):306--315, 1982.

\bibitem{cg84}
J.~T. Cox and S.~Grimmett.
\newblock Central limit theorems for associated random variables and the
  percolation model.
\newblock {\em Ann. Probab.}, 12(2):514--528, 1984.

\bibitem{edwards1971}
R.~D. Edwards and R.~C. Kirby.
\newblock Deformations of spaces of imbeddings.
\newblock {\em Ann. of Math.}, 93:63--88, 1971.

\bibitem{eli85}
A.~I. Elizarov.
\newblock On the variance of the number of stationary points of a homogeneous
  {G}aussian field.
\newblock {\em Theory Probab. Appl.}, 29(3):569--570, 1985.

\bibitem{ef16}
A.~Estrade and J.~Fournier.
\newblock Number of critical points of a {G}aussian random field: condition for
  a finite variance.
\newblock {\em Statist. Probab. Lett.}, 118:94--99, 2016.

\bibitem{el16}
A.~Estrade and J.~R. Le\'{o}n.
\newblock A central limit theorem for the {E}uler characteristic of a
  {G}aussian excursion set.
\newblock {\em Ann. Probab.}, 44(6):3849--3878, 2016.

\bibitem{gm88}
M.~Goresky and R.~MacPherson.
\newblock {\em Stratified {M}orse theory}, volume~14 of {\em Results in
  Mathematics and Related Areas}.
\newblock Springer-Verlag, Berlin, 1988.

\bibitem{js17}
S.~R. Jain and R.~Samajdar.
\newblock Nodal portraits of quantum billiards: domains, lines, and statistics.
\newblock {\em Rev. Modern Phys.}, 89(4):045005, 66, 2017.

\bibitem{kv18a}
M.~Kratz and S.~Vadlamani.
\newblock Central limit theorem for {L}ipschitz-{K}illing curvatures of
  excursion sets of {G}aussian random fields.
\newblock {\em J. Theoret. Probab.}, 31(3):1729--1758, 2018.

\bibitem{kl01}
M.~F. Kratz and J.~R. Le\'{o}n.
\newblock Central limit theorems for level functionals of stationary {G}aussian
  processes and fields.
\newblock {\em J. Theoret. Probab.}, 14(3):639--672, 2001.

\bibitem{kw18}
P.~Kurlberg and I.~Wigman.
\newblock Variation of the {N}azarov-{S}odin constant for random plane waves
  and arithmetic random waves.
\newblock {\em Adv. Math.}, 330:516--552, 2018.

\bibitem{led95}
M.~Ledoux.
\newblock L'alg{\`{e}}bre de {L}ie des gradients it{\'{e}}r{\'{e}}s d'un
  g{\'{e}}n{\'{e}}rateur {M}arkovien -- d{\'{e}}v{\'{e}}loppements de moyennes
  et entropies.
\newblock {\em Ann. Sci. Ecole. Norm. Sup.}, 28(4):435--460, 1995.

\bibitem{mrw20}
D.~Marinucci, M.~Rossi, and I.~Wigman.
\newblock The asymptotic equivalence of the sample trispectrum and the nodal
  length for random spherical harmonics.
\newblock {\em Ann. Inst. Henri Poincar\'{e} Probab. Stat.}, 56(1):374--390,
  2020.

\bibitem{mat73}
J.~N. Mather.
\newblock Stratifications and mappings.
\newblock In {\em Dynamical systems}, pages 195--232. Academic Press, New
  York-London, 1973.

\bibitem{mrv23}
S.~Muirhead, A.~Rivera, and H.~Vanneuville.
\newblock The phase transition for planar {G}aussian percolation models without
  {FKG}.
\newblock {\em Ann. Probab.}, 51(5):1785--1829, 2023.
\newblock With an appendix by L. K\"{o}hler-Schindler.

\bibitem{mul17}
D.~M\"{u}ller.
\newblock A central limit theorem for {L}ipschitz-{K}illing curvatures of
  {G}aussian excursions.
\newblock {\em J. Math. Anal. Appl.}, 452(2):1040--1081, 2017.

\bibitem{an17}
G.~Naitzat and R.~J. Adler.
\newblock A central limit theorem for the {E}uler integral of a {G}aussian
  random field.
\newblock {\em Stoch. Proc. Appl.}, 127(6):2036--2067, 2017.

\bibitem{ns16}
F.~Nazarov and M.~Sodin.
\newblock Asymptotic laws for the spatial distribution and the number of
  connected components of zero sets of {G}aussian random functions.
\newblock {\em J. Math. Phys. Anal. Geom.}, 12(3):205--278, 2016.

\bibitem{ns20}
F.~Nazarov and M.~Sodin.
\newblock Fluctuations in the number of nodal domains.
\newblock {\em J. Math. Phys.}, 61(12):123302, 39, 2020.

\bibitem{nic11}
L.~Nicolaescu.
\newblock {\em An invitation to {M}orse theory}.
\newblock Universitext. Springer, New York, second edition, 2011.

\bibitem{npr19}
I.~Nourdin, G.~Peccati, and M.~Rossi.
\newblock Nodal statistics of planar random waves.
\newblock {\em Comm. Math. Phys.}, 369(1):99--151, 2019.

\bibitem{pen01}
M.~D. Penrose.
\newblock A central limit theorem with applications to percolation, epidemics
  and {B}oolean models.
\newblock {\em Ann. Probab.}, 29(4):1515--1546, 2001.

\bibitem{pit96}
V.~I. Piterbarg.
\newblock {\em Asymptotic methods in the theory of {G}aussian processes and
  fields}, volume 148 of {\em Translations of Mathematical Monographs}.
\newblock American Mathematical Society, Providence, RI, 1996.

\bibitem{pra19}
P.~Pranav et~al.
\newblock Unexpected topology of the temperature fluctuations in the cosmic
  microwave background.
\newblock {\em A\&A}, 627:A163, 2019.

\bibitem{pri20}
L.~Priya.
\newblock Concentration for nodal component count of {G}aussian {L}aplace
  eigenfunctions.
\newblock {\em Ann. Inst. Henri Poincar\'{e} Probab. Stat. (to appear)}.

\bibitem{rv19}
A.~Rivera and H.~Vanneuville.
\newblock Quasi-independence for nodal lines.
\newblock {\em Ann. Inst. Henri Poincar\'{e} Probab. Stat.}, 55(3):1679--1711,
  2019.

\bibitem{slu94}
E.~V. Slud.
\newblock M{WI} representation of the number of curve-crossings by a
  differentiable {G}aussian process, with applications.
\newblock {\em Ann. Probab.}, 22(3):1355--1380, 1994.

\bibitem{tan15}
K.~Tanguy.
\newblock Some superconcentration inequalities for extrema of stationary
  {G}aussian processes.
\newblock {\em Statist. Probab. Lett.}, 106:239--246, 2015.

\bibitem{tv20}
V.~Tassion and H.~Vanneuville.
\newblock Noise sensitivity of percolation via differential inequalities.
\newblock {\em Proc. Lond. Math. Soc. (3)}, 126(4):1063--1091, 2023.

\bibitem{van19}
H.~{Vanneuville}.
\newblock Reading group on random nodal lines at {ETH} {Z}urich.
\newblock
  \url{https://metaphor.ethz.ch/x/2019/hs/401-4600-69L/Reading_group_nodal_1_2.pdf},
  2019.
\newblock Accessed on 25/01/23.

\bibitem{wen05}
H.~Wendland.
\newblock {\em Scattered data approximation}, volume~17 of {\em Cambridge
  Monographs on Applied and Computational Mathematics}.
\newblock Cambridge University Press, Cambridge, 2005.

\bibitem{wor96}
K.~J. Worsley et~al.
\newblock A unified statistical approach for determining significant signals in
  images of cerebral activation.
\newblock {\em Human Brain Mapping}, 4(1):58--73, 1996.

\bibitem{zha01}
Y.~Zhang.
\newblock A martingale approach in the study of percolation clusters on the
  {$\mathbb{Z}^d$} lattice.
\newblock {\em J. Theoret. Probab.}, 14(1):165--187, 2001.

\end{thebibliography}

\end{document}